\documentclass[oneside,reqno,12pt,a4paper]{amsart}
\usepackage{amsfonts,latexsym,amsmath,amssymb,amsthm}
\usepackage{graphicx}
\usepackage{caption}
\usepackage{xcolor}
\usepackage[noadjust]{cite}
\usepackage{footnote}
\usepackage[foot]{amsaddr}
\usepackage{txfonts}

\def\a{\alpha}
\newcommand\be{\beta}

\newcommand\De{\Delta}
\def\th{\theta}

\newcommand\et{\eta}
\newcommand\ep{\epsilon}

\newcommand\ga{\gamma}

\newcommand\la{\lambda}

\newcommand\ph{\varphi}

\newcommand\rh{\rho}
\newcommand\si{\sigma}

\newcommand\ta{\tau}

\newcommand\Om{\Omega}

\newcommand\A{{\mathcal A}}
\newcommand\B{{\mathcal B}}

\newcommand\Reals{\mathbb R}
\newcommand\wb{w}
\newcommand{\sgn}{\operatorname{sgn}}
\newcommand\DA{D_A}
\newcommand\DB{D_B}
\newcommand\underlineDA{\underline{D}_A}
\newcommand\overlineDA{\overline{D}_A}
\newtheorem{theorem}{Theorem}
\newtheorem{lemma}[theorem]{Lemma}
\newtheorem{remark}[theorem]{Remark}

\newtheorem{proposition}[theorem]{Proposition}
\newtheorem{corollary}[theorem]{Corollary}

\newtheorem*{proposition*}{Proposition}

\newtheorem{definition}{Definition}
\numberwithin{equation}{section}
\numberwithin{theorem}{section}
\begin{document}
\parindent=15pt
\normalbaselineskip18pt
\baselineskip18pt
\global\hoffset=-15truemm
\global\voffset=-10truemm
\allowdisplaybreaks[2]
\title[Two-locus clines in a heterogeneous environment]{Two-locus clines maintained by diffusion and recombination in a heterogeneous environment}

\author{Linlin Su$^{1,\ast}$}
\author{King-Yeung Lam$^{2,\ast}$}
\author{Reinhard B\"{u}rger $^{3,\ddagger}$}

\footnotetext[1]{Department of Mathematics, Southern University of Science and Technology, Shenzhen, P. R. China}
\footnotetext[2]{Department of Mathematics,
Ohio State University, Columbus, OH 43210, USA}
\footnotetext[3]{Department of Mathematics, University of Vienna, 1090 Vienna, Austria}
\renewcommand{\thefootnote}{\fnsymbol{footnote}}
\footnotetext[1]{These authors contributed equally}
\footnotetext[3]{Corresponding author}
\renewcommand{\thefootnote}{\arabic{footnote}}

\email{sull@sustc.edu.cn, lam.184@osu.edu, reinhard.buerger@univie.ac.at}

\date{\today}

\begin{abstract}
We study existence and stability of stationary solutions of a system of semilinear parabolic partial differential equations that occurs in population genetics. It describes the evolution of gamete frequencies in a geographically structured population of migrating individuals in a bounded habitat. Fitness of individuals is determined additively by two recombining, diallelic genetic loci that are subject to spatially varying selection. Migration is modeled by diffusion. Of most interest are spatially non-constant stationary solutions, so-called clines. In a two-locus cline all four gametes are present in the population, i.e., it is an internal stationary solution. We provide conditions for existence and linear stability of a two-locus cline if recombination is either sufficiently weak or sufficiently strong relative to selection and diffusion. For strong recombination, we also prove uniqueness and global asymptotic stability. For arbitrary recombination, we determine the stability properties of the monomorphic equilibria, which represent fixation of a single gamete.
\end{abstract}
\maketitle



\newpage

\setcounter{section}{0}

\section{Introduction}
A cline describes a gradual change in genotypic or phenotypic frequency as a function of spatial location. Clines often occur in species distributed along an environmental gradient, for instance in temperature, where alternative phenotypes or genotypes are better adapted to the different extremes of the environment. They are frequently observed in natural populations and are important objects of research in evolutionary biology and ecology (e.g.\ \cite{Adrion_etal2015}, \cite{Bedford_etal2015}, \cite{Endler1977}). Measurements of their shape admit inferences about the relative strength of migration and selection.

The mathematical theory of clines was initiated by Haldane \cite{Haldane1948}, who derived a reaction-diffusion equation for the equilibrium allele frequencies at a diallelic locus subject to spatially varying selection along a single spatial dimension. He computed the cline, the spatially non-constant solution, for special cases. The mathematical theory of clines became a very active research area in the 1970s, when the consequences of various assumptions about spatial variation in fitnesses and about migration patterns were investigated (Slatkin \cite{Slatkin1973}, Nagylaki \cite{N1975,N1976,N1978}). These authors derived parabolic partial differential equations to describe and study not only the allele frequencies at equilibrium, but also their evolution. At about the same time, and motivated by this work, Conley \cite{Conley1975}, Fleming \cite{Fleming1975}, Fife and Peletier \cite{Fife&Peletier1977,Fife&Peletier1981}, and Henry \cite{Henry1981} developed and employed advanced mathematical methods to investigate existence, uniqueness, and stability of clinal solutions under a variety of assumptions about fitnesses. We refer to spatially nonuniform stationary solutions of the parabolic PDE as clines.
More recently, Lou, Nagylaki, and their collaborators \cite{LN2002,LN2004,LN2006,LNS2010,Nakashima2016,Nakashima2018,NNS2010}  extended previous work in several directions by modeling migration by general elliptic operators on bounded domains in arbitrary dimensions, by admitting wide classes of fitness functions, by including dominance, and by studying multiallelic loci. Several of these extensions revealed qualitatively new features. The theory of one-locus clines has been reviewed in \cite{NL2008} and \cite{LNN2013}.

In the present work, we study two-locus clines. Understanding their properties is of biological relevance because many traits are determined by multiple genetic loci which undergo recombination. The resulting mathematical models are much more complex than one-locus models, because the interaction of selection and migration generates probabilistic associations (correlations) among these loci, so called linkage disequilibria, which are eroded in turn by recombination. We shall focus on the simplest case of two diallelic loci with additive fitnesses. The first study of a two-locus cline model is due to Slatkin \cite{Slatkin1975}, who showed numerically that the linkage disequilibrium generated between the two loci tends to steepen the cline. Barton \cite{Barton1983,Barton1986} derived some general results about the consequences of linkage on the linkage disequilibria among multiple loci and provided numerical results that can guide intuition. Most recently, B\"urger \cite{RB2017} analysed a two-locus model in which, following Haldane \cite{Haldane1948}, simple step functions are used to describe the spatial dependence of fitnesses along the real line. Using a singular-perturbation approach, an explicit approximation of the two-locus cline was obtained for the case of strong recombination. The steepening of the cline by linkage could be proved and quantified.

Our aim here is to develop a rigorous mathematical theory for the existence, uniqueness, and stability of two-locus clines on bounded domains in $\Reals^n$ for fitnesses depending on the spatial location in a general way. In Section 2, we introduce the basic model, which is formulated as a system of semilinear parabolic PDEs. In Section 3, we collect several preliminaries that will be used subsequently. Section 4 is devoted to the study of the boundary equilibria. These can be monomorphic equilibria, i.e., constant stationary solutions such that both loci are globally fixed for one allele, or clines at one locus with the second locus fixed for one or the other allele. For the monomorphic equilibria, stability and bifurcations are determined.

In Section 5, we investigate the case of no recombination. The results follow from the theory of diallelic and multiallelic one-locus models \cite{LN2002,LN2004,LN2006} and provide the basis for the investigation of clines maintained under weak recombination, which is the topic of Section 6. There, existence of an asymptotically stable two-locus cline is proved based on a regular perturbation argument. Finally, in Section 7, we treat strong recombination. This may be the biologically most frequently realized case because it applies when the loci are located on different chromosomes or on the same chromosome, but not close together. We prove existence, uniqueness, and global stability of a two-locus cline. In addition to standard elliptic and parabolic PDE methods, our proofs invoke perturbation techniques, persistence, and dynamical systems theory.
The article closes by a brief discussion and by mentioning some open problems.

\section{Model}\label{sec:model}
We consider a monoecious, diploid population that occupies a bounded, open domain $\Om\subset\Reals^n$ with $C^2$ boundary $\partial\Om$. Fitness of individuals depends on location, but is independent of time, population density, or genotype frequencies. It is determined by two diallelic loci, $\A$ and $\B$, which recombine at rate $r\ge0$. We model migration by diffusion and assume it is homogeneous, isotropic, and genotype-independent. If the migration variance is $\si^2$, the diffusion constant is $d=\tfrac12\si^2$ \cite{N1975,N1989}.

If the alleles at locus $\A$ are denoted by $A$ and $a$, and those at $\B$ by $B$ and $b$, then there are the four possible gametes $AB$, $Ab$, $aB$, and $ab$, which we label as $i=1$, 2, 3, and 4, respectively. We write $I=\{1,2,3,4\}$ for the set of gametes.
Let the frequency of gamete $i$ at position $x\in\Om$ and time $t$ be $p_i=p_i(x,t)$, where $p_i\ge0$ and $\sum_{i=1}^4 p_i=1$, and let $p=(p_1,p_2,p_3,p_4)^T$. We denote the usual measure of linkage disequilibrium by
\begin{equation}\label{eq:D}
	D = D(p) = p_1p_4-p_2p_3\,.
\end{equation}
If $w_{ij}(x)$ is the fitness of the diploid genotype $ij$ at location $x\in\Om$, then
\begin{equation}\label{w}
	w_i = w_i(x,p) = \sum_{j=1}^4 w_{ij}(x)p_j \;\text{ and }\; \wb = \wb(x,p) = \sum_{i=1}^4 w_ip_i
\end{equation}
are the marginal fitness of gamete $i$ and  the population mean fitness, respectively. As is biologically reasonable and common, throughout we posit $w_{ij}=w_{ji}$ and $w_{14}=w_{23}$, i.e., absence of position effects, and assume that every $w_{ij}$ is real valued and H\"older continuous, i.e., $w_{ij}\in C^\ga(\bar\Omega)$ for some $\ga\in(0,1)$.

\subsection{Evolutionary equations}
We assume that (i) the three evolutionary forces selection, migration, and recombination are of the same order of magnitude and sufficiently weak, (ii) migration is genotype independent, spatially uniform, and isotropic, and (iii) individuals mate locally at random so that Hardy-Weinberg proportions are obtained locally. By approximating the exact discrete-space discrete-time model (\cite{RB2009}, \cite{N2009}) by a continuous-space continuous-time model as in \cite{N1989}, the evolution of the gamete frequencies $p_i$, $i\in I$, is described by the following system of partial differential equations:
\begin{subequations}\label{dynamics_dsr_pi}
\begin{alignat}{2}
	&\partial_t p_i = d\De p_i + s S_i(x,p) - \et_i r D   &\quad&\text{ for } (x,t)\in\Om\times(0,\infty) \,, \label{dynamics_dsr_pi_a} \\
	&\partial_\nu p_i = 0 	&\quad&\text{ for } (x,t)\in\partial\Om\times(0,\infty)\,,  \label{dynamics_dsr_pi_b} \\
	&p_i(x,0)\ge0 \, \text{ and }  \sum_{i=1}^4 p_i(x,0)=1 &\quad&\text{ for $x\in\bar\Om$} \label{dynamics_dsr_pi_c}
\end{alignat}
\end{subequations}
(cf.\ \cite{Slatkin1975, LNN2013, RB2017}).
Here, $\De$ is the Laplace operator in $\Reals^n$, $d>0$ the diffusion constant, $s>0$ a measure of the strength of selection, $r\ge0$ the recombination rate,
\begin{equation}\label{et}
	\et_1=\et_4=-\et_2=-\et_3 = 1\,,
\end{equation}
and $\nu$ is the unit outer normal vector to the boundary $\partial\Om$. The terms $\et_i r D$ describe the effects of recombination (see Section \ref{sec:reco_LD}). The functions
\begin{equation}\label{S_i2}
	S_i(x,p) = p_i(w_i-\wb)
\end{equation}
arise from selection (see Section \ref{sec:assumptions_selection}).

In many situations, it will be more convenient to scale away $d$ because we focus on the role of recombination. Therefore, if we fix $d>0$ and set $\la=s/d$, $\rho=r/d$, rescale time according to $\ta=td$, and return to $t$ instead of $\ta$, we can rewrite \eqref{dynamics_dsr_pi} as
\begin{subequations}\label{dynamics_pi}
\begin{alignat}{2}
	&\partial_t p_i = \De p_i + \la S_i(x,p) - \et_i \rh D   &\quad&\text{ for } (x,t)\in\Om\times(0,\infty) \,, \label{dynamics_pi_a} \\
	&\partial_\nu p_i = 0 	&\quad&\text{ for } (x,t)\in\partial\Om\times(0,\infty)\,,  \label{dynamics_pi_b} \\
	&p_i(x,0)\ge0 \, \text{ and }  \sum_{i=1}^4 p_i(x,0)=1 &\quad&\text{ for $x\in\bar\Om$}\,. \label{dynamics_pi_c}
\end{alignat}
\end{subequations}

\subsection{Basic properties of the dynamics}\label{sec:basic_dynamics}
If the initial data $p_i(x,0)$ are continuous on $\bar\Om$, then \eqref{dynamics_pi} has a unique classical solution $p(x,t)$ for every $\rh\ge0$ that exists for all $t\ge0$. It satisfies
\begin{equation}\label{constraint1}
	p_i(x,t)\ge0 \text{ and }  \sum_{i=1}^4 p_i(x,t)=1 \; \text{ on } \bar\Om\times(0,\infty)\,.
\end{equation}
In addition, if for some $i\in I$,
\begin{equation}\label{flow_into_interior1}
	p_i(x,0) \not\equiv0 \text{ on } \bar\Om, \text{ then }
		p_i(x,t) > 0 \text{ on } \bar\Om\times(0,\infty)\,.
\end{equation}
The first assertion in \eqref{constraint1} and \eqref{flow_into_interior1} follow from the strong maximum principle for parabolic equations \cite{PW}. For the second assertion in \eqref{constraint1}, we observe from \eqref{w}, \eqref{et}, \eqref{S_i2}, and \eqref{dynamics_pi_a} that
\begin{equation}\label{PDE_sum_p_i}
	\partial_t \left(\sum_{i=1}^4 p_i\right) = \De \left(\sum_{i=1}^4 p_i\right)  + \la \wb\left(1-\sum_{i=1}^4 p_i\right)\,.
\end{equation}
Therefore, uniqueness of solutions of \eqref{PDE_sum_p_i} yields $\sum_{i=1}^4 p_i(x,t)=1$ (see \cite{LNN2013}).

We define
\begin{equation}\label{X}
	\mathbf{X}=\biggl\{(u_1,u_2,u_3,u_4) \in C(\bar\Omega; [0,1]^4):\; \sum_{i=1}^4 u_i \equiv 1\biggr\}\footnote{We write $C(\bar\Omega;S)$ for the space of $S$-valued uniformly continuous functions on $\bar\Omega$ equipped with the supremum norm, and $C(\bar\Omega)=C(\bar\Omega;\Reals)$.}
\end{equation}
and
\begin{equation}\label{eq:X0}
	\mathbf{X}_0=\{ (u_1,u_2,u_3,u_4) \in \mathbf{X}:\; u_1+u_2 \equiv 0  \text{ or } u_3+u_4 \equiv 0  \text{ or } u_1 + u_3 \equiv 0  \text{ or } u_2+u_4 \equiv 0\}\,,	
\end{equation}
where $\mathbf{X}_0$ is the subset of $\mathbf{X}$ that corresponds to fixation (across the whole population) of at least one of the alleles at one of the loci.
We define $\Psi$ to be the semiflow generated by \eqref{dynamics_pi} in $\mathbf{X}$, i.e., for initial data $U_0\in \mathbf{X}$ and every $t > 0$ we set $\Psi_t(U_0)=p(\cdot,t)$, where $p(\cdot,t)$ is the solution of \eqref{dynamics_pi} corresponding to $p(\cdot,0)=U_0(\cdot)$. The above considerations show that $\mathbf{X}$ is positively invariant under the flow $\Psi$. It is easily seen that each of the four `edges' in $\mathbf{X}_0$ is invariant. In addition, we have the following property.

\begin{lemma}\label{flow_into_interior}
If $\rh>0$, then $\Psi$ maps $\mathbf{X}\setminus \mathbf{X}_0$ into the interior of $\mathbf{X}$.
\end{lemma}

\begin{proof}
It is sufficient to consider the flow on the boundary of $\mathbf{X}$. By \eqref{flow_into_interior1}, it is sufficient to assume $p_i(x,0) \equiv 0$ for some $i$. By symmetry, we need to consider only the case $p_1(x,0) \equiv 0$.
Because $p(\cdot,0)\notin\mathbf{X}_0$, $p_1(x,0) \equiv 0$ implies the existence of $x_2, x_3\in\Omega$ such that $p_2(x_2,0)>0$ and $p_3(x_3,0)>0$. Then, again by the maximum principle for parabolic equations (and because of Neumann boundary conditions), $p_2(x,t)>0$ and $p_3(x,t)>0$ on $\bar\Omega\times(0,\infty)$. Now, we argue by contradiction to show that $p_1(x,t)>0$ on $\bar\Omega\times(0,\infty)$. Suppose that $p_1(x_1,t_1)=0$ for some $x_1\in\bar\Om$ and $t_1>0$. Then $S_1(x_1,p(x_1,t_1))=0$. If $x_1\in\Om$, then $\partial_t p_1(x_1,t_1)\le0$ and $\De p_1(x_1,t_1)\ge0$, which contradicts
\begin{equation}
	\partial_t p_1(x_1,t_1) - \Delta p_1(x_1,t_1) = \rh\, p_2(x_1,t_1)p_3(x_1,t_1) >0\,.
\end{equation}
This leaves us with the case $x_1\in\partial\Om$ and $p_1(x,t_1)>0$ for all $(x,t)\in\Omega\times(0,\infty)$, for which the Hopf lemma shows that $\partial_\nu p_1(x_1,t_1)<0$. This contradicts \eqref{dynamics_pi_b}. Therefore, $p_1(x,t)$ is positive on $\bar\Om$ whenever $t>0$.
\end{proof}

\subsection{Properties of recombination and linkage disequilibrium}\label{sec:reco_LD}
The measure $D$ of linkage disequilibrium can be interpreted as the covariance of the random variables indicating presence or absence of allele $A$ ($B$) at locus $\A$ ($\B$). Indeed, from \eqref{constraint1} we deduce
\begin{equation}\label{D_cov}
	D = p_1p_4 - p_2 p_3 = p_1(p_1+p_2+p_3+p_4) - (p_1+p_2)(p_1+p_3) = p_{AB} - p_A p_B\,,
\end{equation}
where $p_{AB}=p_1$, and
\begin{equation}\label{pApB}
	p_A=p_1+p_2 \text{ and } p_B=p_1+p_3
\end{equation}
denote the frequencies of alleles $A$ and $B$, respectively. In particular, recombination erodes linkage disequilibrium because, in the absence of diffusion and selection, $\partial_t D = \et_i \partial_t p_i = -\rh D$ for every $i\in I$, as we easily derive from \eqref{D_cov} and \eqref{dynamics_pi_a}. Recombination also generates missing gametes. For instance, if $p_1(x,0)=0$, but $p_2(x,0)>0$ and $p_3(x,0)>0$, then recombination will generate gamete $AB$ immediately, i.e., $p_1(x,t)>0$ for $t>0$ (see also Lemma \ref{flow_into_interior}). Consult \cite{Geiringer1944} and \cite{LewontinKojima1960} for important early treatments of linkage disequilibrium, and to \cite{Slatkin2008} for its applications in modern genetics.

If recombination is absent, i.e., $\rh=0$, then alleles on the same gamete are never separated and therefore each gamete $i\in I$ may be regarded as an allele at a single locus. Thus, the system \eqref{dynamics_pi} reduces to a one-locus system with four alleles. This case is treated in Section~\ref{sec:no_rec}.

If recombination is strong relative to selection and diffusion, then rapid decay of linkage disequilibrium $D$ to values close to zero will occur. In the limiting case of $D\equiv0$, i.e., vanishing covariance, the loci become independent. In Section \ref{sec:strong_reco}, we treat the case $\rh\gg1$ as a perturbation of that of two independent loci.

\subsection{Assumptions on selection}\label{sec:assumptions_selection}
Concerning selection, which arises as a consequence of a spatially heterogeneous environment, we assume that both loci are subject to so called additive selection, i.e., we ignore dominance and epistasis. Therefore, we can assign the Malthusian parameters $\tfrac12\a(x)$ and $-\tfrac12\a(x)$ to the alleles $A$ and $a$, and $\tfrac12\be(x)$ and $-\tfrac12\be(x)$ to $B$ and $b$, where $\a(x)$ and $\be(x)$ are real-valued functions on $\bar\Om$. They reflect the influence of environmental heterogeneity on the fitnesses of the alleles. Then the fitness coefficients of the gametes $AB$, $Ab$, $aB$, $ab$ are
\begin{alignat}{2}\label{s1234}
	s_1(x) &=\tfrac12 [\a(x)+\be(x)] \,, \quad& s_2(x) &=\tfrac12 [\a(x)-\be(x)] \,, \notag \\
	s_3(x) &=\tfrac12 [-\a(x)+\be(x)] \,, \quad& s_4(x) &=-\tfrac12 [\a(x)+\be(x)] \,,
\end{alignat}
respectively, and the genotypic fitnesses are $w_{ij}(x)=s_i(x)+s_j(x)$. Using $\sum_i p_i(x,t)=1$, straightforward calculations yield
\begin{subequations}\label{S_i}
\begin{align}
	S_1(x,p) &= p_1[\a(x)(p_3+p_4)+\be(x)(p_2+p_4)]\,, \\
	S_2(x,p) &= p_2[\a(x)(p_3+p_4)-\be(x)(p_1+p_3)]\,,  \\
	S_3(x,p) &= p_3[-\a(x)(p_1+p_2)+\be(x)(p_2+p_4)]\,, \\
	S_4(x,p) &= p_4[-\a(x)(p_1+p_2)-\be(x)(p_1+p_3)]\,.
\end{align}
\end{subequations}

Throughout this paper, we will study \eqref{dynamics_pi}, or the equivalent  \eqref{dynamics_dsr_pi}, by assuming \eqref{S_i}. In addition, the following assumption will play an important role:

\noindent {\bf (A)} The functions $\a(x)$ and $\beta(x)$ change sign in $\Omega$ and are of class $C^\ga(\bar\Omega)$ for some $\ga\in(0,1)$.

\section{Preliminaries}\label{sec:prelim}
\subsection{Eigenvalue problems with indefinite weight}\label{sec:ev-problems}
The linearized problem of \eqref{dynamics_pi} at an equilibrium $\hat{p}=(\hat{p}_1, \hat{p}_2, \hat{p}_3, \hat{p}_4)^T$,
$\hat{p}_i=\hat{p}_i(x)$, reads
\begin{subequations}\label{1.4}
\begin{alignat}{2}
	&\De \Phi+J|_{\hat{p}} \Phi +\mu \Phi=0 &\quad&\text{in } \Omega\,, \label{1.4a}\\
	&\partial_\nu\Phi=0		&\quad&\text{on } \partial\Omega\,, \label{1.4b}
\end{alignat}
\end{subequations}
where $\Phi=(\phi_1,\phi_2,\phi_3,\phi_4)^T$, $\phi_i=\phi_i(x)$,
$\sum_{i=1}^{4}\phi_i=0$, and
\begin{align}
	J= & \la
	\begin{pmatrix}
		0 & \beta p_1 & \alpha p_1 & (\alpha+\beta) p_1\\
		-\beta p_2 & 0 & (\alpha-\beta) p_2 & \alpha p_2\\
		-\alpha p_3 & (\beta-\alpha)p_3 & 0 & \beta p_3\\
		-(\alpha+\beta) p_4 & -\alpha p_4 & -\beta p_4 & 0
	\end{pmatrix} \notag \\
	&+ \rho
	\begin{pmatrix}-p_4 & p_3 & p_2 & -p_1 \\
		p_4 & -p_3 & -p_2 & p_1 \\ p_4 & -p_3 & -p_2 & p_1 \\ -p_4 & p_3 & p_2 & -p_1
	\end{pmatrix} \nonumber\\
	& + \la~\mbox{diag} \{\alpha (p_3+p_4)+\beta (p_2+p_4),~\alpha
	(p_3+p_4)-\beta (p_1+p_3), \nonumber\\
	& \qquad -\alpha (p_1+p_2)+\beta (p_2+p_4),~-\alpha (p_1+p_2)-\beta
	(p_1+p_3) \}\,.  \label{Jacob}
\end{align}
Sometimes it is more convenient to study \eqref{1.4} with three linearly
independent equations using the relation $\sum_{i=1}^{4}\phi_i=0$.

For any function $u(x)\in C(\bar{\Omega})$, we define its spatial average
\begin{equation}\label{1.9}
	\bar{u}=\frac{1}{|\Omega|}\int_{\Omega} u(x)~dx\,.
\end{equation}
The following eigenvalue problem will be helpful:
\begin{subequations}\label{1.6}
\begin{alignat}{2}
	&\De \varphi+\tilde{\la} h(x)\varphi=0   &\quad&\text{in } \Om\,, \label{1.6a}\\
	&\varphi>0 		&\quad&\text{in } \Om\,, \label{1.6b}\\
	&\partial_\nu\varphi = 0  &\quad&\text{on } \partial\Om\,, \label{1.6c}
\end{alignat}
\end{subequations}
where $\Omega$ and $\nu$ are as in \eqref{dynamics_pi} and $h(x)\in C(\bar{\Omega})$.
Brown and Lin \cite{BL80} showed that \eqref{1.6} has a positive eigenvalue $\tilde{\la}$ if and only if $h(x)$
changes sign and $\bar{h}<0$. In addition, the positive eigenvalue (if it exists)
is unique, and we denote it by $\la^*(h)$.

For each fixed $\tilde{\la}>0$, we consider the eigenvalue problem
\begin{subequations}\label{1.7}
\begin{alignat}{2}
	&\De \psi+\tilde{\la} h(x)\psi+\mu\psi=0 	&\quad&\text{in } \Om \,, \label{1.7a}\\
	&\partial_\nu\psi = 0 		&\quad&\text{on } \partial\Omega\,, \label{1.7b}
\end{alignat}
\end{subequations}
where $\Omega$ and $\nu$ are as in \eqref{dynamics_pi} and $h(x)\in C(\bar{\Omega})$.

The following results are well known (\cite{S83}, \cite{LNN2013}).
\begin{lemma}\label{lm1.2} Suppose that $h(x)\in C(\bar{\Omega})$ is a nonconstant function and positive somewhere. Then the smallest eigenvalue $\mu_1(\tilde{\la})$ of \eqref{1.7}  is strictly concave down in $\tilde{\la}$,
\begin{equation}\label{mu1}
\lim_{\tilde{\la}\to \infty}\mu_1(\tilde{\la})=-\infty\,,
\end{equation}
and has the following properties.

\noindent{\rm (a)} If $\bar{h}\ge 0$,
then $\mu_1(\tilde{\la})<0$ and $\mu_1(\tilde{\la})$ is strictly decreasing for $\tilde{\la}>0$.

\noindent{\rm (b)} Assume that $\bar{h}<0$. Then
\begin{equation}\label{1.8}
	\mu_1(\tilde{\la})\begin{cases} <0 &\mbox{if} \;\;
	\tilde{\la}>\la^*(h)\,,\\
	 =0 &\mbox{if}  \;\; \tilde{\la}=\la^*(h)\,,\\
	 >0 &\mbox{if}  \;\; 0<\tilde{\la}<\la^*(h)\,, \end{cases}
\end{equation}
and $\mu_1(\tilde{\la})$ is strictly decreasing for
$\tilde{\la}>\la^*(h)$.
\end{lemma}

\begin{remark}\label{rm1.3} {\rm Because the eigenfunction corresponding to
$\mu_1(\tilde{\la})$ can be chosen to be positive on $\Omega$, integration of
\eqref{1.7a} over $\Omega$ shows that if $h(x)\le 0$ and $h(x)\not\equiv 0$,
then $\mu_1(\tilde{\la})>0$ for every $\tilde{\la}>0$.}
\end{remark}

For a nonconstant function $h(x)\in C(\bar{\Omega})$, it is convenient to define
\begin{equation}\label{lambda_0}
	\la_{0}(h) =
	\begin{cases} \la^*(h) &\mbox{if $h(x)$ changes sign and} \; \bar h <0\,,\\
		 0 					&\mbox{if}  \; \bar h \ge 0\,,\\
		 \infty &\mbox{if}  \;  h(x)\le 0\mbox{~in $\bar\Omega$}\,.
	\end{cases}
\end{equation}
Then Lemma \ref{lm1.2} and Remark \ref{rm1.3} yield

\begin{lemma}\label{rm1.4} Suppose that $h(x)$ is a nonconstant continuous function on $ \bar{\Omega}$.
If $\tilde{\la}>\la_{0}(h)$, then
$\mu_1(\tilde{\la})<0$ and $\mu_1(\tilde{\la})$ is strictly decreasing in $\tilde{\la}$.
If $0<\tilde{\la}<\la_{0}(h)$, then $\mu_1(\tilde{\la})>0$.
\end{lemma}

\subsection{One-locus theory}\label{sec:SLClines}
The diallelic one-locus equation with isotropic, homogeneous migration, and selection without dominance reads
\begin{subequations}
\begin{align}\label{eq:theta}
		&\partial_t \theta  = \Delta \theta+\la h(x) \theta (1-\theta) 	&&\text{for } (x,t) \in \Omega\times(0,\infty)\,, \\
	&\partial_\nu \theta =0 	&&\text{for }(x,t) \in \partial\Omega\times(0,\infty)\,, \\
	&\theta(x,0) = \theta_0(x) 	&&\text{for }x \in \Omega \text{ and }\theta_0\in C^0(\bar\Omega;[0,1])\setminus\{0,1\}\,.
\end{align}
\end{subequations}
Recalling that $\la^\ast(h)$ designates the unique positive eigenvalue of \eqref{1.6}, for a sign-changing $h(x)$ we introduce
\begin{equation}\label{lambda_h}
	\la_{h}:=
	\begin{cases} \la^*(h) &\mbox{if } \; \bar h <0\,,\\
		 0 					&\mbox{if }  \; \bar h =0\,,\\
		 \la^*(-h) &\mbox{if }  \; \bar h >0\,.
	\end{cases}
\end{equation}

\begin{theorem}[{\cite[Lemma\,10.1.5]{Henry1981}}, {\cite[Theorem\,2.1]{LN2002}}]\label{thm:singlelocus}
Let $h(x)$ be a sign-changing function of class $C^\gamma(\bar\Omega)$ for some $0 < \gamma < 1$. Then for every $\la>0$, the problem \eqref{eq:theta} has a unique stable equilibrium solution $\theta_h$, and every solution $\theta(x,t)$ converges to $\theta_h(x)$ uniformly in $x$ as $t \to \infty$. More precisely:

\noindent {\rm (a)} Suppose that $\bar h <0$. If $0 < \la \leq \la_h$, then $\theta_h \equiv 0$ in $\bar\Omega$; if $\la > \la_h$, then $0 <\theta_h < 1$ in $\bar\Omega$.

\noindent {\rm (b)} Suppose that $\bar h >0$. If $0 < \la \leq \la_h$, then $\theta_h \equiv 1$ in $\bar\Omega$; if $\la > \la_h$, then $0 <\theta_h < 1$ in $\bar\Omega$.

\noindent {\rm (c)} Suppose that $\bar h = 0$. Then for every $\la >0$, $0 < \theta_h < 1$ in $\bar\Omega$.

\noindent In each case, $\theta_h$ is linearly stable whenever $\la\neq\la_h$. The proof of Theorem 2.1 in \cite{LN2002} shows that convergence occurs in $C^2(\bar\Om)$.
\end{theorem}

For convenience, we call the constant equilibria $\theta(x)\equiv 0$ and $\theta(x)\equiv 1$ in $\bar\Omega$ the {\it trivial equilibria}, and we call $\theta_h$ the {\it global attractor} of \eqref{eq:theta}. If $0<\theta_h<1$, then we call it a (one-locus) {\it cline}.

\section{Boundary equilibria}\label{Section_boundary}
\subsection{Existence}\label{sec:existence_boundary}
The four monomorphic equilibria $M_i$, defined by $p_i\equiv1$, exist always. We also call them the vertices or vertex equilibria.

In addition, \eqref{dynamics_pi} may have up to six equilibria on the edges connecting any pair of vertices.
We define
\begin{equation}\label{1.20}
	h_{ij}(x)=s_i(x)-s_j(x)\,.
\end{equation}
Let
$\hat{p}^{(ij)} = \hat{p}^{(ij)}(x)$, $i<j$, be the edge equilibrium with gametes $i$ and $j$ present, i.e.,
\begin{equation}\label{2.4}
	\hat{p}^{(ij)}_k = \begin{cases} \theta_{ij} &\mbox{if} \;\; k=i\,,\\
	 1-\theta_{ij} &\mbox{if}  \;\; k=j\,,\\
	 0 &\mbox{if}  \;\; k\neq i,j\,, \end{cases}
\end{equation}
where $\theta_{ij}=\theta_{ij}(x)$ satisfies
\begin{subequations}\label{2.5}
\begin{alignat}{2}
	&\De \theta_{ij}+\la h_{ij}(x)\theta_{ij}(1-\theta_{ij})=0 &\quad&\text{in } \Om\,, \label{2.5a}\\
	&0<\theta_{ij}<1 		&\quad&\text{in } \Om\,, \label{2.5b}\\
	&\partial_\nu\theta_{ij} = 0     &\quad&\text{on } \partial\Om\,. \label{2.5c}
\end{alignat}
\end{subequations}
Theorem~\ref{thm:singlelocus} and the above-cited result of Brown and Lin \cite{BL80} for \eqref{1.6} inform us that \eqref{2.5} has a solution if and only if
\begin{subequations}\label{2.6}
\begin{equation}\label{2.6a}
	h_{ij}(x)~\mbox{changes sign in $\Omega$}
\end{equation}
and
\begin{equation}\label{2.6b}
	\la >\la_{ij}:=\la_{h_{ij}}\,,
\end{equation}
\end{subequations}
where $\la_{h_{ij}}$ is given by \eqref{lambda_h} with $h=h_{ij}$.
Moreover, if a solution of \eqref{2.5} exists, it is unique and linearly stable.

If $\rh=0$, then all six edge equilibria may exist. If $\rh>0$, then only $\hat p^{(12)}$, $\hat p^{(13)}$, $\hat p^{(34)}$, and $\hat p^{(24)}$ can exist (Lemma \ref{flow_into_interior}). These four edge equilibria are independent of $\rh$ because $D\equiv0$ at each of them; see also Section \ref{sec:single-locus_poly}. The biological reason for the non-existence of $\hat p^{(14)}$ and $\hat p^{(23)}$ if $\rh>0$ is that recombination generates the two other gametes immediately (cf.\ Section \ref{sec:reco_LD}).

\subsection{Stability of the monomorphic equilibria}
Here we show that generically at most one monomorphic equilibrium can be linearly stable. In Theorem~\ref{thm:monos_stability}, we determine the range of parameters for which it is stable. For sufficiently strong migration (relative to selection and recombination), we establish global asymptotic stability in Theorem~\ref{thm:M1_globallystable}.

We write
\begin{equation}\label{2.17}
	I_i=I\setminus \{i\}\,,
\end{equation}
and define, for each fixed $j\in I$,
\begin{equation}\label{3.2}
	\tilde{j}=5-j\,,\quad \tilde{I}_j=I\setminus\{j,\tilde{j}\}\,,
\end{equation}
i.e., $\tilde{I}_1=\tilde{I}_4=\{2,3\}$ and $\tilde{I}_2=\tilde{I}_3=\{1,4\}$.

From \eqref{s1234} we observe that
\begin{equation}\label{s14_s23}
	s_1(x)= -s_4(x) \quad\text{and}\quad s_2(x)= -s_3(x)
\end{equation}
for every $x\in \Omega$. Therefore, there are only two possibilities:
\begin{subequations}\label{3.8}
\begin{gather}
	\mbox{there exists $i\in I$ such that}\;
		\bar{s}_{\tilde{i}}<\bar{s}_{k}<\bar{s}_{i}\; \mbox{for each}\;
		k\in\tilde{I}_i\,;  \label{3.8a}\\
	\mbox{there exist $i\in I$ and $j\in \tilde{I_i}$ such that}\;
		\bar{s}_{i}=\bar{s}_{j}=-\bar{s}_{\tilde{i}}=-\bar{s}_{\tilde{j}}\,. \label{3.8b}
\end{gather}
\end{subequations}
We note that \eqref{3.8a} is the generic case, which is equivalent to
\begin{equation}\label{bar_si_max}
	\mbox{there exists an $i\in I$ such that $\bar{s}_i>\max_{j\in I_i}\{\bar{s}_j\}$\,.}
\end{equation}

To study the stability of the vertex equilibrium $M_j$, we have to investigate the eigenvalue problem
\begin{subequations}\label{3.3}
\begin{alignat}{2}
		&\De \phi_i+\la h_{ij}(x)\phi_i +\rho \phi_{\tilde{j}}+\mu \phi_i=0
				&\quad&\text{in } \Omega\,, \label{3.3a}\\
		&\De \phi_{\tilde{j}}+\la h_{\tilde{j}j}(x)\phi_{\tilde{j}} -\rho \phi_{\tilde{j}}+\mu \phi_{\tilde{j}}=0
				&\quad&\text{in } \Omega\,,\label{3.3b} \\
		&\partial_{\nu} \phi_i=\partial_{\nu}\phi_{\tilde{j}}=0
				&\quad&\text{on }\partial\Omega\,,\label{3.3c}
\end{alignat}
\end{subequations}
where $i\in \tilde{I}_j$ (cf.\ \cite[(2.23)]{LN2006} and \eqref{Jacob}). For each $k\in I$, we let $E_{k}$ be the set of all
eigenvalues of the single-equation eigenvalue problem
\begin{subequations}\label{3.4}
\begin{alignat}{2}
	&\De \phi^{(k)}+\la h_{kj}(x)\phi^{(k)} +\mu^{(k)} \phi^{(k)}=0 	&\quad&\text{in } \Omega\,, \label{3.4a} \\
	&\partial_{\nu} \phi^{(k)}=0 	&\quad&\text{on } \partial\Omega\,. \label{3.4b}
\end{alignat}
\end{subequations}

Before formulating and proving our main results, we establish two lemmas.

\begin{lemma}\label{lm3.1}
For every $\rho\ge 0$ and every $j\in I$ fixed, the set of eigenvalues of system \eqref{3.3} consists of
$\bigcup_{i\in \tilde{I}_j} E_i\,  \bigcup \,\{\mu^{(\tilde{j})}+\rho: \mu^{(\tilde{j})}\in E_{\tilde{j}}\}$.
\end{lemma}

\begin{proof} First, we observe that for every $i\in \tilde{I}_j$, every
$\mu^{(i)}\in E_i$ with an eigenfunction $\phi^{(i)}$ is also an eigenvalue of
\eqref{3.3}
and the corresponding eigenfunction has components $\phi_i=\phi^{(i)}$ and
$\phi_k\equiv 0$ for $k\neq i$.
Second, for every
$\mu^{(\tilde{j})}\in E_{\tilde{j}}$ with an eigenfunction $\phi^{(\tilde{j})}$,
there are two cases. If $\mu^{(\tilde{j})}+\rho\in E_i$ for some $i\in
\tilde{I}_j$,
then we already know it is an eigenvalue of \eqref{3.3} from the above
discussion. If $\mu^{(\tilde{j})}+\rho\notin E_i$ for every $i\in \tilde{I}_j$,
then the operator
\begin{equation}\label{3.5}
	L_i:=\left\{\De +\la h_{ij}(x)+\mu^{(\tilde{j})}+\rho\right\}
\end{equation}
is invertible for every $i\in \tilde{I}_j$, whence $\mu^{(\tilde{j})}+\rho$ is an
eigenvalue of \eqref{3.3} whose eigenfunction has components
\begin{subequations}\label{3.6}
\begin{alignat}{2}
	&\phi_i = L_i^{-1}[-\rho \phi^{(\tilde{j})}] &\quad&\text{for } i\in \tilde{I}_j\,, \label{3.6a}\\
	&\phi_{\tilde{j}}=\phi^{(\tilde{j})}\,.&&  \label{3.6b}
\end{alignat}
\end{subequations}

Next, we show that if $\mu$ is an eigenvalue of \eqref{3.3}, then either $\mu\in
E_{i}$ for some $i\in \tilde{I}_j$
or $\mu=\mu^{(\tilde{j})}+\rho$ for some $\mu^{(\tilde{j})}\in E_{\tilde{j}}$.
We denote the components of the eigenfunction of $\mu$ by $\phi_{i}$ for $i\in I_j$. There are two possibilities.
If $\phi_{\tilde{j}}\equiv 0$, then there exists at least one
$\phi_{i}\not\equiv 0$, $i\in \tilde{I}_j$, whence in view of \eqref{3.3a} we
conclude that
$\mu\in E_i$ and the corresponding eigenfunction can be taken as
$\phi^{(i)}=\phi_{i}$.
If $\phi_{\tilde{j}}\not\equiv 0$, then from \eqref{3.3b} we see that
$\mu=\mu^{(\tilde{j})}+\rho$ for some $\mu^{(\tilde{j})}\in E_{\tilde{j}}$ and
the corresponding
eigenfunction can be chosen as $\phi^{(\tilde{j})}=\phi_{\tilde{j}}$. This
completes the proof of Lemma \ref{lm3.1}.
\end{proof}

For a fixed $j\in I$, let $\mu_1^{(\tilde{j})}(\la)$ be the smallest eigenvalue of \eqref{3.4} with $k=\tilde{j}$. From
Lemma \ref{rm1.4} and \eqref{lambda_0}, we see that if $0\le \la_0(h_{\tilde{j}j})<\infty$, then
for $\la>\la_0(h_{\tilde{j}j})$ we have
$\mu_1^{(\tilde{j})}(\la)<0$ and $\mu_1^{(\tilde{j})}(\la)$ is strictly
decreasing. Thus, for each $\rho>0$, there exists a unique $\la$, denoted by
$\la_0(h_{\tilde{j}j},\rho)$, such that $\la>\la_0(h_{\tilde{j}j})$ and $\mu_1^{(\tilde{j})}(\la)+\rho=0$.
If $\la_0(h_{\tilde{j}j})=\infty$, we define $\la_0(h_{\tilde{j}j},\rho)=\infty$. If $\rho=0$, we set $\la_0(h_{\tilde{j}j},0)=\la_0(h_{\tilde{j}j})$. Then, for $\rho\ge 0$, we have
$\mu_1^{(\tilde{j})}(\la)+\rho<0$ if $\la>\la_0(h_{\tilde{j}j},\rho)$ and
$\mu_1^{(\tilde{j})}(\la)+\rho>0$ if $0<\la<\la_0(h_{\tilde{j}j},\rho)$.

Now, for every $j\in I$ and $\rho\ge 0$, we define
\begin{subequations}\label{3.7}
\begin{align}
	\la_j^*(\rho) &=\min_{i\in \tilde{I}_j}\{\la_0(h_{ij}),\la_0(h_{\tilde{j}j},\rho)\}\,, \label{3.7a}\\
	\mu_j^* &=\min_{i\in \tilde{I}_j}\{\mu_1^{(i)}(\la),\mu_1^{(\tilde{j})}(\la)+\rho\}\,. \label{3.7b}
\end{align}
\end{subequations}
The above discussion and Lemma \ref{rm1.4} inform us that $\mu_j^*>0$ if
$0<\la<\la_j^*(\rho)$ and $\mu_j^*<0$ if $\la>\la_j^*(\rho)$.
Since Lemma \ref{lm3.1} reveals that
$M_j$ is stable if $\mu_j^*>0$ and unstable if $\mu_j^*<0$, we have proved the following.

\begin{lemma}\label{lm3.2}
Let $\rh\ge 0$.

\noindent {\rm (a)} If $\la_j^*(\rho)=0$, then $M_j$ is linearly
unstable for every $\la>0$.

\noindent {\rm (b)} If $0<\la_j^*(\rho)<\infty$, then $M_j$ is linearly stable for
$0<\la<\la_j^*(\rho)$ and linearly unstable for $\la>\la_j^*(\rho)$.

\noindent {\rm (c)} If $\la_j^*(\rho)=\infty$, then $M_j$ is linearly stable for every $\la>0$.
\end{lemma}
Notice that if $\rho=0$, then the conclusions in Lemma \ref{lm3.2} are established in \cite[p.\,637]{LN2006}.

\begin{remark}\label{rm3.3}\rm
If $h_{\tilde{j}j}(x)\equiv 0$ and $\rho>0$, then from \eqref{3.4a} with $k=\tilde{j}$ we see
that $\mu_1^{(\tilde{j})}(\la)=0$ for every $\la>0$ and thus
$\mu_1^{(\tilde{j})}(\la)+\rho>0$ for every $\la>0$. Therefore, when
$h_{\tilde{j}j}(x)\equiv 0$ and $\rho>0$, we set $\la_0(h_{\tilde{j}j},\rho)=\infty$ and
the conclusions in Lemma \ref{lm3.2} still hold.
\end{remark}

\begin{theorem}\label{thm:monos_stability}
Suppose (A) and that \eqref{bar_si_max} holds for some $i\in I$. Then we have for every $\rho\ge 0$:

\noindent {\rm (a)} Every $M_j$ other than $M_i$ is linearly unstable.

\noindent {\rm (b)} Let $\la_i^*(\rho)$ be given by \eqref{3.7a}. Then $0<\la_i^*(\rho)<\infty$ and $M_i$ is linearly stable if $0<\la<\la_i^*(\rho)$; $M_i$ is linearly unstable if $\la>\la_i^*(\rho)$.
\end{theorem}

If $\rho=0$, Theorem \ref{thm:monos_stability} follows directly from Theorem 1.5 in \cite{LN2006}. Its proof inspired the following proof.

\begin{proof}
(a) For each $j\neq i$, there are two cases. If $j\neq \tilde{i}$, i.e, $i\neq
\tilde{j}$, by \eqref{bar_si_max} we have $\bar{s}_i>\bar{s}_j$,
whence we obtain $\la_0(h_{ij})=0$ from \eqref{lambda_0} and \eqref{1.20}. Therefore,
\eqref{3.7a} yields $\la_{j}^*(\rho)=0$.
If $j=\tilde{i}$, by \eqref{3.8a} we have $\bar{s}_k>\bar{s}_j$ and hence $\la_0(h_{kj})=0$ for $k\in \tilde{I}_i=\tilde{I}_j$.
Therefore, \eqref{3.7a} implies again that $\la_{j}^*(\rho)=0$. Now we deduce from Lemma \ref{lm3.2}(a) that $M_j$ is unstable for every $\la>0$, which proves part (a).

(b) In view of \eqref{bar_si_max} and \eqref{lambda_0}, we have
$\la_0(h_{ki})>0$ for every $k\in\tilde{I}_i$. From \eqref{s1234} we observe
that
\begin{equation}\label{3.9}
	s_m(x)-s_{l}(x)\in \{\pm\alpha(x),\pm\beta(x)\}\quad \mbox{for every } l\in
	I \mbox{ and every } m\in\tilde{I}_l\,.
\end{equation}
Since both $\alpha(x)$ and $\beta(x)$ change sign and $k\in\tilde{I}_i$, it
follows from \eqref{3.9} and \eqref{lambda_0} that $\la_0(h_{ki})<\infty$.
On account of the definition of $\la_0(h_{\tilde{i}i},\rho)$ we have $\la_0(h_{\tilde{i}i},\rho)>0$. Then \eqref{3.7a}
implies that $0<\la_{i}^*(\rho)<\infty$ and part (b) follows immediately
from Lemma \ref{lm3.2}(b).
\end{proof}

\begin{remark}\label{rm3.5}\rm
Because $\mu_1^{(\tilde{j})}(\la)$ is strictly decreasing for $\la>\la_0(h_{\tilde{j}j})$ by Lemma \ref{lm1.2}(b), the critical value
$\la_0(h_{\tilde{j}j},\rho)$ is strictly increasing in $\rho$ by its definition. Therefore, \eqref{3.7a}
implies that $\la_j^*(\rho)$ is nondecreasing in $\rho$. Thus, Theorem
\ref{thm:monos_stability}(b) shows that increasing the recombination rate facilitates
stability of the monomorphic equilibrium with the highest spatially averaged fitness.
\end{remark}

\begin{theorem}\label{thm:M1_globallystable}
Suppose (A) and that \eqref{bar_si_max} holds for some $i\in I$. Then, for every fixed $r\ge 0$ and $s>0$, there exists $d_0=d_0(r,s)\gg1$ such that $M_i$ is globally asymptotically stable for \eqref{dynamics_dsr_pi} if $d>d_0$.
\end{theorem}

\begin{proof}
The proof is based on Theorem 2.1 in \cite{LN2006}. We set
\begin{equation}\label{3.10}
 T_i(x,p)=s S_i(x,p)-\eta_i r D(p)\,,\quad i\in I\,.
\end{equation}
Then the spatially averaged system (2.3) of \cite{LN2006} becomes
\begin{subequations}\label{3.11}
\begin{equation}
	\frac{dq^*_i}{d\tau}=s \bar{S}_i(q^*)-\eta_i r D(q^*), \label{3.11a}
\end{equation}
\begin{equation}
	q^*(0)\in \mbox{int}\, \Delta_4\,, \label{3.11b}
\end{equation}
\end{subequations}
where
\begin{equation}\label{3.13}
	\Delta_4:=\{p\in \Reals^4: p_i\ge 0~\mbox{for every}~i\in I,
	~\sum_{j=1}^{4}p_j=1\}\,,
\end{equation}
\begin{subequations}\label{3.12}
\begin{align}
	\bar{S}_1(q^*) &= q^*_1[\bar{\alpha}(q^*_3+q^*_4)+\bar{\beta}(q^*_2+q^*_4)]\,, \label{3.12a} \\
	\bar{S}_2(q^*) &= q^*_2[\bar{\alpha}(q^*_3+q^*_4)-\bar{\beta}(q^*_1+q^*_3)]\,, \label{3.12b} \\
	\bar{S}_3(q^*) &= q^*_3[-\bar{\alpha}(q^*_1+q^*_2)+\bar{\beta}(q^*_2+q^*_4)]\,,\label{3.12c}\\
	\bar{S}_4(q^*) &= q^*_4[-\bar{\alpha}(q^*_1+q^*_2)-\bar{\beta}(q^*_1+q^*_3)]\,. \label{3.12d}
\end{align}
\end{subequations}

The system of ODEs \eqref{3.11} describes the dynamics in a simple two-locus model without migration,
epistasis, or dominance. Therefore, mean fitness is a global Lyapunov function \cite{Ewens1969}. Hence,
every solution of \eqref{3.11} converges to an equilibrium. In addition, every
equilibrium $q^*$ of \eqref{3.11} is in linkage equilibrium, i.e., it satisfies $D(q^*)=0$ (\cite{Lyubich92}, \cite{NHB1999}).

We are informed by \eqref{s1234}, \eqref{s14_s23}, and \eqref{bar_si_max} that $\bar{\alpha}\neq 0$ and $\bar{\beta}\neq
0$, whence it is clear from \eqref{3.12} that the only solutions to $\bar S_j(q^*)=0$ for every $j\in I$ are the
monomorphic equilibria $M_j$. Simple analysis of the linearized problem of \eqref{3.11} at
each $M_j$ shows that if \eqref{bar_si_max} holds for some $i\in I$, then $M_i$ is the only linearly
stable monomorphic equilibrium. The other
monomorphic equilibria are all unstable; they may have stable manifolds, but the stable manifolds are either
invariant edges corresponding to a marginal one-locus system or connect to the vertices from the exterior of the state space $\Delta_4$. Therefore, every solution of \eqref{3.11} converges to $M_i$.

Thus, we have shown that (A4) in \cite{LN2006} holds with $\hat{q}^*=M_i$. Therefore, Theorem 2.1 in \cite{LN2006} applies and, together
with statement (b), yields the global asymptotic stability of $M_i$ with respect to the full system \eqref{dynamics_dsr_pi} provided $d\gg1$.
\end{proof}

\begin{remark}\label{rem:4.7}\rm
Because the critical value $d_0$ originating from Theorem 2.1 in \cite{LN2006} may depend on $r$ and $s$, we cannot conclude that for every fixed $\rh\ge0$, there exists a $\la_0 \ll 1$ such that $M_i$ is globally asymptotically stable for \eqref{dynamics_pi} if $\la < \la_0$. However, we conjecture that it is true.
\end{remark}

In the nongeneric case \eqref{3.8b}, we obtain the following result.

\begin{proposition}\label{th3.6}
Suppose that (A) and \eqref{3.8b} hold. Then, for every $\rho\ge 0$, all monomorphic equilibria are linearly unstable.
\end{proposition}

\begin{proof} In view of \eqref{3.8b}, \eqref{3.9}, and \eqref{lambda_0}, for the $i,j$ in \eqref{3.8b}, we
have
\begin{equation}
	\la_0(h_{ji})=\la_0(h_{ij})=\la_0(h_{\tilde{j}\tilde{i}})=\la_0(h_{\tilde{i}\tilde{j}})=0\,.
\end{equation}
We conclude from \eqref{3.7a} that $\la_k^*(\rho)=0$ for every
$\rho\ge 0$ and every $k\in I$.  From Lemma \ref{lm3.2}(a),
we infer that for every $\rho\ge 0$ each $M_k$ is unstable for every $\la>0$.
\end{proof}

\subsection{Equilibria with one polymorphic locus}\label{sec:single-locus_poly}
From \eqref{s14_s23} we obtain $h_{12}=h_{34}$ and $h_{13}=h_{24}$. Therefore, the edge equilibria $\hat{p}^{(12)}$ and $\hat{p}^{(34)}$ as well as $\hat{p}^{(13)}$ and $\hat{p}^{(24)}$ exist only pairwise, i.e., if one member of a pair exists then also the other. We call them single-locus polymorphisms, or single-locus clines, because at each of these equilibria one locus maintains both alleles at positive frequency, whereas at the other locus one allele is fixed. For instance, $\hat{p}^{(12)}(x)$ describes a cline at locus $\B$ with allele $A$ fixed at locus $\A$. It is well known that a one-locus cline is globally asymptotically stable within its edge (Theorem \ref{thm:singlelocus}). However, determining stability of these equilibria with respect to the full system \eqref{dynamics_pi} is a challenging task and has been resolved only for special cases (see below).

\section{No recombination}\label{sec:no_rec}
In this section, we treat the case $r=0$, i.e., $\rho=0$. Therefore, the results depend only on $s/d=\la$, and we use \eqref{dynamics_pi} throughout. Because $\rho=0$, we may regard each gamete $i\in I$ as an allele at one locus. Therefore, the system \eqref{dynamics_pi} simplifies to a one-locus four-allele model, and the results of Lou and Nagylaki \cite{LN2002,LN2004,LN2006} on multiallelic one-locus
models apply. We consider various assumptions on the functions $\alpha(x)$ and $\beta(x)$ and start with the most specific and simplest scenario that is of biological interest.

\subsection{The functions $\alpha(x)$ and $\beta(x)$ have the same spatial dependence}
We assume that
\begin{subequations}\label{2.1}
\begin{equation}\label{2.1a}
	\alpha (x)=a g(x)\,,\;\beta (x)=b g(x)\,,
\end{equation}
where
\begin{equation}\label{2.1b}
	\mbox{the constants $a$ and $b$ are positive and the function $g(x)$ changes sign.}
\end{equation}
\end{subequations}
Then \eqref{s1234} reduces to
\begin{eqnarray}\label{2.2}
	s_1(x)=\tfrac{1}{2}(a+b)g(x)\,,\,&&
	s_2(x)=\tfrac{1}{2}(a-b)g(x)\,,\,\nonumber\\
	s_3(x)=\tfrac{1}{2}(b-a)g(x)\,,\,&&
	s_4(x)=-\tfrac{1}{2}(a+b)g(x)\,.
\end{eqnarray}

By \eqref{2.1}, the conditions (A2) and (A3) in \cite{LN2002} hold with
\begin{subequations}
\begin{equation}\label{2.7}
	\sigma (x)= h_{14}(x)=(a+b)g(x)\,,
\end{equation}
\begin{equation}\label{2.8}
	\gamma_2=a(a+b)^{-1}\,,\quad \gamma_3=b(a+b)^{-1}\,.
\end{equation}
\end{subequations}
Therefore, we obtain the following results directly from Theorems 3.2 and 3.3 in \cite{LN2002}.

\begin{proposition}\label{prop2.1}
If $\rho=0$ and \eqref{2.1} holds, system \eqref{dynamics_pi} has always a globally attracting equilibrium.

\noindent {\rm (a)} Suppose that $\bar{g}<0$. Then $(0,0,0,1)^{T}$ is globally
asymptotically stable if $0<\la\le\la^*(\sigma)$, and
$\hat{p}^{(14)}$ is globally asymptotically stable if
$\la>\la^*(\sigma)$.

\noindent {\rm (b)} Suppose that $\bar{g}>0$. Then $(1,0,0,0)^{T}$ is globally
asymptotically stable if $0<\la\le \la^*(-\sigma)$, and
$\hat{p}^{(14)}$ is globally asymptotically stable if
$\la>\la^*(-\sigma)$.

\noindent {\rm (c)} Suppose that $\bar{g}=0$. Then $\hat{p}^{(14)}$ is
globally asymptotically stable for every $\la>0$.
\end{proposition}

\subsection{The functions $\alpha(x)$ and $\beta(x)$ have the same sign}
We assume that
\begin{equation}\label{2.9}
	\beta (x)=\alpha (x)\gamma(x)\,,\; \mbox{where $\gamma(x)>0$ for every $x\in
	\bar{\Omega}$}.
\end{equation}
Then \eqref{s1234} reduces to
\begin{eqnarray}\label{2.10}
	 s_1(x)=\tfrac{1}{2}(1+\gamma(x))\alpha(x)\,,&&
	 s_2(x)=\tfrac{1}{2}(1-\gamma(x))\alpha(x)\,,\\
	 s_3(x)=\tfrac{1}{2}(\gamma(x)-1)\alpha(x)\,,&&
	 s_4(x)=-\tfrac{1}{2}(1+\gamma(x))\alpha(x)\,.\nonumber
\end{eqnarray}

The following result follows directly from Remark 3.3 in \cite{LN2006}. We present
a proof here using the idea mentioned there.

\begin{proposition}\label{prop2.3}
Assume that $\rho=0$, that the function $\alpha(x)$ changes sign, and that
\eqref{2.9} holds. Then $\hat{p}^{(14)}$ is globally asymptotically stable
for $\la\gg 1$.
\end{proposition}

\begin{proof}
By \eqref{2.9} and \eqref{2.10}, we have
\begin{subequations}\label{2.11}
\begin{equation}\label{2.11a}
	 s_2(x)<\max_{j\neq 2} s_j(x)\;\; \mbox{and} \;\; s_3(x)<\max_{j\neq 3}
	s_j(x)\quad \mbox{for every $x\in\bar{\Omega}$}\,,
\end{equation}
\begin{equation}\label{2.11b}
	 s_1(x)>\max_{j\neq 1} s_j(x)\;\;\mbox{when $\alpha(x)>0$\;\; and}\;\;
	s_4(x)>\max_{j\neq 4} s_j(x)\;\; \mbox{when $\alpha(x)<0$}\,.
\end{equation}
\end{subequations}
Let $p=(p_1, p_2, p_3, p_4)^T$ be any solution of \eqref{dynamics_pi}. Therefore, for
$\la$ sufficiently large, \eqref{2.11a} and \cite[Corollary 4.7]{LN2004} imply
that
\begin{equation}\label{2.12}
	p_i(x,t)\to 0  \;\; \mbox{uniformly in $x$ as $t\to \infty$ for $i=2,3$\,.}
\end{equation}
By \eqref{2.11b} and \cite[Corollary 4.9]{LN2004}, for $i=1, 4$, there exists
$\delta_i^*=\delta_i^*(\la)>0$
such that for all initial data that satisfy \eqref{dynamics_pi_c}, there exists $t_i^*$,
which may depend on $\la$
and the initial data, such that
\begin{equation}\label{2.15}
p_i(x,t)\ge \delta_i^*  \;\; \mbox{for every $x\in\bar{\Omega}$ and every
$t\ge t_i^*$.}
\end{equation}

Now pick any sequence $\{t_k\}_{k=1}^{\infty}$ such that $t_k\to \infty$ as
$k\to \infty$.
The estimate \cite[(3.19)]{LN2002} shows that, passing to a subsequence if
necessary, $p(x,t_k)\to \hat{p}(x)$ as $k\to\infty$ in
$C^2(\bar{\Omega})$, where $\hat{p}$ is an equilibrium of system \eqref{dynamics_pi}. Then from \eqref{2.12} and \eqref{2.15} we conclude that $\hat{p}_i(x)=0$
for $i=2, 3$ and $\hat{p}_i(x)\ge\delta_i^*$ for $i=1, 4$, respectively. Since
the only equilibrium with the
gametes 1 and 4 present, and 2 and 3 absent, is $\hat{p}^{(14)}$ (see \eqref{2.4} -- \eqref{2.6}), we must have $\hat{p}=\hat{p}^{(14)}$.
Therefore, the $\omega$-limit set of any initial data that satisfies \eqref{dynamics_pi_c} is
$\{\hat{p}^{(14)}\}$, and hence $p(x,t)\to \hat{p}^{(14)}(x)$ as $t\to\infty$.

Finally, from \eqref{2.9} and \eqref{2.10} we observe that
\begin{equation}
	\max [s_2(x),s_3(x)]<\max [s_1(x),s_4(x)]\quad\mbox{for every
	$x\in\bar{\Omega}$}\,,
\end{equation}
whence Theorem 1.6 in \cite{LN2006} informs us that $\hat{p}^{(14)}$ is
asymptotically stable for
$\la$ sufficiently large. This completes the proof.
\end{proof}

\subsection{Arbitrary functions $\alpha(x)$ and $\beta(x)$}
We recall the definition of $I_i$ from \eqref{2.17}
and make the generic assumption that \eqref{bar_si_max} holds for some $i\in I$.
Then \cite[Theorem 1.1]{LN2006} yields

\begin{proposition}\label{prop2.4} Assume that $\rho=0$.
Let $p=(p_1, p_2, p_3, p_4)^T$ denote an arbitrary solution of \eqref{dynamics_pi} with $p_i(x,0)\not\equiv 0$. Then for
$0<\la \ll 1$, as $t\rightarrow\infty$, $p_i(x,t)\to 1$ uniformly in $x$.
\end{proposition}

\begin{remark}\rm From \eqref{s1234} we see that \eqref{bar_si_max} holds with
\begin{equation}\label{2.13}
i=\begin{cases} 1 &\mbox{if} \;\;\bar{\alpha}>0,~\bar{\beta}>0 \,,\\
	2 &\mbox{if}  \;\; \bar{\alpha}>0,~\bar{\beta}<0\,,\\
	3 &\mbox{if}  \;\; \bar{\alpha}<0,~\bar{\beta}>0\,,\\
	4 &\mbox{if}  \;\; \bar{\alpha}<0,~\bar{\beta}<0\,.
\end{cases}
\end{equation}
\end{remark}

\begin{remark} \label{rm2.5}{\rm
We observe that if neither $\bar{\alpha}$ nor $\bar{\beta}$ is zero (as in the
four cases in \eqref{2.13}), then
$\bar{s}_j\neq \bar{s}_k$ for every $j\neq k$. Therefore, if $\rho=0$, then according to \cite[Remark 1.3]{LN2006},
for sufficiently small $\la$, the vertices are the only equilibria of \eqref{dynamics_pi}.}
\end{remark}

As $\la$ increases, the edge equilibria will appear if (A) holds. The next result determines the stability of each of them immediately after its appearance \cite[Theorem 1.7]{LN2006}; the notation $\la_{ij}$ is as in \eqref{2.6b}.

\begin{proposition}\label{prop:no_rec_general}
Suppose that $\rho=0$, that each of the functions $\alpha(x)$, $\beta(x)$,
$\alpha(x)+\beta(x)$, and $\alpha(x)-\beta(x)$ changes sign, and that assumption
\eqref{bar_si_max} holds for some $i\in I$.

\noindent {\rm (a)} There exists $\delta_1>0$ such that $\hat{p}^{(jk)}$ is linearly
unstable if $j,k\in I_i$, $j<k$, and
$\la_{jk}<\la<\la_{jk}+\delta_1$.

\noindent {\rm (b)} Suppose further that $\la_{ik}<\min_{j\in I_i,j\neq
k}\la_{ij}$ for some
$k\in I_i$. Then there exists $\delta_2>0$ such that $\hat{p}^{(ik)}$ is
linearly stable if $\la_{ik}<\la<\la_{ik}+\delta_2$,
and $\hat{p}^{(il)}$ is linearly unstable if $l\neq k$ and
$\la_{il}<\la<\la_{il}+\delta_2$.
\end{proposition}

\begin{remark}\label{rem:edge_equil}\rm
Suppose that $i$ is the gamete with the highest spatially averaged fitness. Under the assumption in Proposition~\ref{prop:no_rec_general}(b), we infer from \eqref{3.7a} that $\la_{i}^*(\rho)=\la_{ik}$. Then Theorem~\ref{thm:monos_stability}(b) shows that $M_i$ is linearly stable if $0<\la<\la_{ik}$ and unstable if $\la>\la_{ik}$.
Proposition~\ref{prop:no_rec_general}(b) informs us that as $\la$ increases from $0$, $\hat p^{(ik)}$ is the first one that moves into the state space among the edge equilibria that bifurcate through $M_i$, and initially it is linearly stable (by exchange of stability with $M_i$). All other edge equilibria that may move into the state space will be unstable immediately after their appearance.
\end{remark}

If there exists $x_i\in \Omega$ for $i=1,2,3,4$ such that
\begin{equation}\label{2.14}
\alpha(x_1)\,,\, \beta(x_1)>0\,;\;\alpha(x_2)>0\,,\, \beta(x_2)<0\,;\;
\alpha(x_3)<0\,,\, \beta(x_3)>0\,;\;\alpha(x_4)\,,\, \beta(x_4)<0\,,
\end{equation}
then Corollary 4.10 in \cite{LN2004} guarantees the existence of an internal
equilibrium for $\la\gg 1$.

\begin{proposition}\label{prop2.7}
Suppose that $\rho=0$ and that \eqref{2.14} holds. Then for $\la
\gg 1$, system \eqref{dynamics_pi} has at least one equilibrium
$\hat{p}=(\hat{p}_1, \hat{p}_2, \hat{p}_3, \hat{p}_4)^T$ such that
$\hat{p}_i(x)>0$ in $\Omega$ for every $i$.
\end{proposition}

\section{Weak recombination}\label{sec:weak_rec}
Here, we study \eqref{dynamics_dsr_pi} for weak recombination, i.e., $d$ and $s$ are fixed and $0<r\ll1$. This is equivalent to studying \eqref{dynamics_pi} with $\la>0$ fixed and $0<\rh\ll1$, which we use henceforth.
From Section \ref{Section_boundary}, we already know that the four single-locus polymorphisms
$\hat{p}^{(12)}$, $\hat{p}^{(34)}$, $\hat{p}^{(13)}$, and $\hat{p}^{(24)}$, defined by \eqref{2.4} and \eqref{2.5}, exist in pairs and neither their values nor their existence depends on $\rh$. This is different for the edge equilibria $\hat{p}^{(14)}$ and $\hat{p}^{(23)}$, which can exist only if $\rh=0$. Suppose that $\hat{p}^{(14)}$ (or $\hat{p}^{(23)}$) exists when $\rh=0$. If we increase $\rho$ from $0$ slightly, will $\hat{p}^{(14)}$ ($\hat{p}^{(23)}$) move into the interior of the state space $\mathbf{X}$ and therefore become full polymorphisms? The investigation of this problem is the main purpose of this section. Throughout, we suppose assumption (A). Our main result is the following.

\begin{theorem}\label{thm:weak_reco}
\noindent {\rm (a)} If for $\rho=0$ the edge equilibrium $\hat{p}^{(14)}$ $(\hat{p}^{(23)})$ exists and is linearly stable, then for every sufficiently small $\rho>0$, problem \eqref{dynamics_pi} has an internal equilibrium $\hat{p}^{(\rho)}$ that is linearly stable, and $\hat{p}^{(\rho)}(x)\to \hat{p}^{(14)}(x)$ $(\hat{p}^{(23)}(x))$ uniformly as $\rho\to 0+$.

\noindent {\rm (b)} Assume that each of $\alpha(x)$, $\beta(x)$, $\alpha(x)+\beta(x)$, and $\alpha(x)-\beta(x)$ changes sign, \eqref{bar_si_max} holds for $i=1$, and
$\la_{14}<\min\{\la_{12},\la_{13}\}$. Then there exists $\delta>0$ such that
for every $\la\in (\la_{14}, \la_{14}+\delta)$ and every
sufficiently small $\rho>0$, problem \eqref{dynamics_pi} has an internal equilibrium $\hat{p}^{(\rho)}$, which is linearly stable. Moreover, for every fixed $\la\in (\la_{14}, \la_{14}+\delta)$, we have $\hat{p}^{(\rho)}(x)\to \hat{p}^{(14)}(x)$ uniformly as $\rho\to 0+$.
\end{theorem}

\begin{remark}\rm
1. Note that the assumption \eqref{bar_si_max} for $i=1$ can be imposed without loss of generality upon relabeling of gametes.

\noindent {\rm 2.} Recall from Proposition \ref{prop:no_rec_general} and Remark \ref{rem:edge_equil} that for $\rh=0$, $\la_{14}$ is the critical eigenvalue at which $\hat{p}^{(14)}$ appears by an exchange-of-stability bifurcation with $M_1$ as $\la$ increases above $\la_{14}$. Moreover, $\la_{14}<\min\{\la_{12},\la_{13}\}$ implies that $\hat{p}^{(14)}$ appears before the two pairs of edge equilibria ($\hat{p}^{(12)}$ and $\hat{p}^{(34)}$, $\hat{p}^{(13)}$ and $\hat{p}^{(24)}$) as $\la$ increases from 0.
\end{remark}

To prove Theorem \ref{thm:weak_reco}, we need some preparations.
Recalling \eqref{1.4}, \eqref{Jacob}, \eqref{1.20}, \eqref{2.4}, \eqref{2.5}, and using $\sum_{i=1}^{4}\phi_i=0$, the linearized problem of \eqref{dynamics_pi} with $\rho=0$ at $\hat{p}^{(14)}(x)$ reads
\begin{subequations}\label{4.1}
\begin{alignat}{2}
	&\De\phi_1+\la h_{14}(1-2\theta_{14})\phi_1 -\la\theta_{14}[h_{24}\phi_2+h_{34}\phi_3]+\mu \phi_1=0
			&\quad&\text{in } \Omega\,, \label{4.1a} \\
	&\De\phi_2+\la (h_{24}-h_{14}\theta_{14})\phi_{2} +\mu\phi_2=0 &\quad&\text{in } \Omega\,, \label{4.1b} \\
	&\De\phi_3+\la (h_{34}-h_{14}\theta_{14})\phi_{3} +\mu\phi_3=0 &\quad&\text{in } \Omega\,, \label{4.1c} \\
	&\partial_\nu \phi_i = 0\,, \quad i=1,2,3, &\quad&\text{on } \partial\Omega\,. \label{4.1d}
\end{alignat}
\end{subequations}
There are three single-equation linearized problems related to \eqref{4.1}:
\begin{equation}\label{4.5}
	\De\phi^{(1)}+\la h_{14}(1-2\theta_{14})\phi^{(1)} +\mu^{(1)}\phi^{(1)}=0\quad
	\mbox{in $\Omega$\,,} \quad \partial_\nu \phi^{(1)}=0 \quad \mbox{on~} \partial\Omega\,.
\end{equation}
\begin{equation}\label{4.4}
	\De\phi^{(2)}+\la (h_{24}-h_{14}\theta_{14})\phi^{(2)} +\mu^{(2)}\phi^{(2)}=0\quad
	\mbox{in $\Omega$\,,} \quad \partial_\nu \phi^{(2)}=0 \quad \mbox{on~} \partial\Omega\,.
\end{equation}
\begin{equation}\label{4.8}
	\De\phi^{(3)}+\la (h_{34}-h_{14}\theta_{14})\phi^{(3)} +\mu^{(3)}\phi^{(3)}=0\quad
	\mbox{in $\Omega$\,,} \quad \partial_\nu \phi^{(3)}=0 \quad \mbox{on~} \partial\Omega\,.
\end{equation}
We denote the set of eigenvalues of \eqref{4.1}, \eqref{4.5}, \eqref{4.4}, and \eqref{4.8} by
$E$, $E^{(1)}$, $E^{(2)}$, and $E^{(3)}$, respectively.

\begin{lemma}\label{lm4.1}
The set of eigenvalues of problem \eqref{4.1} consists of the eigenvalues of problems
\eqref{4.5}, \eqref{4.4}, and \eqref{4.8}, namely, $\displaystyle E=\bigcup_{i=1}^{3} E^{(i)}$.
\end{lemma}

\begin{proof}
First, we show that $\displaystyle E \supseteq \bigcup_{i=1}^{3} E^{(i)}$. Suppose $\mu^{(1)}\in E^{(1)}$
 with an eigenfunction $\phi^{(1)}$, then it is clear that $\mu^{(1)}$ solves \eqref{4.1} with
 $\phi_1=\phi^{(1)}$, $\phi_2=0$, and $\phi_3=0$, and therefore $\mu^{(1)}\in E$.
If $\mu^{(2)}\in E^{(2)} \setminus E^{(1)}$ with an eigenfunction $\phi^{(2)}$, we see that it is also an eigenvalue
of \eqref{4.1} by taking $\phi_2=\phi^{(2)}$, $\phi_3=0$, and solving $\phi_1$ from
\eqref{4.1a}. Similarly, if $\mu^{(3)}\in E^{(3)}\setminus E^{(1)}$ with an eigenfunction $\phi^{(3)}$, we see that it is also an eigenvalue
of \eqref{4.1} by taking $\phi_2=0$, $\phi_3=\phi^{(3)}$, and solving $\phi_1$ from
\eqref{4.1a}.

Second, we demonstrate the converse $\displaystyle E \subseteq \bigcup_{i=1}^{3} E^{(i)}$. If $\mu$ is an eigenvalue
of \eqref{4.3} with $\phi_2=\phi_3=0$, then $\phi_1\neq 0$ and therefore $\mu$
is an eigenvalue of \eqref{4.5}; otherwise, if $\phi_2\neq 0$ or $\phi_3\neq 0$, then
$\mu$ is an eigenvalue of \eqref{4.4} or \eqref{4.8}, respectively.

Thus, the set of eigenvalues of \eqref{4.3}
consists the eigenvalues of \eqref{4.5}, \eqref{4.4}, and \eqref{4.8}.
\end{proof}

\begin{proof}[Proof of Theorem \ref{thm:weak_reco}]
(a) We present the proof only for $\hat{p}^{(14)}$; for $\hat{p}^{(23)}$ it is similar.

By the asumption that $\hat{p}^{(14)}$ is linearly stable when $\rho=0$,
every $\mu$ that satisfies \eqref{4.1} has a positive real part unless $\phi_i\equiv 0$ for $i=1,2,3$. Therefore, by the implicit
function theorem, there exists a family of equilibria $\hat{p}^{(\rho)}$ for $\rho>0$ sufficiently small and
$\hat{p}^{(\rho)}(x)\to \hat{p}^{(14)}(x)$ uniformly as $\rho\to 0+$. From \eqref{1.4} and \eqref{Jacob} we
infer that the linearization of \eqref{dynamics_pi} at $\hat{p}^{(\rho)}$ is a small continuous perturbation of \eqref{4.1} for which every eigenvalue also has a positive real part, whence $\hat{p}^{(\rho)}$ is linearly stable.

Next, we show that $\hat{p}^{(\rho)}$ is in the interior of $\mathbf{X}$. By the fact that $\hat{p}^{(14)}_1(x)>0$ and
$\hat{p}^{(14)}_4(x)>0$ in $\bar{\Omega}$ and the uniform continuity of $\hat{p}^{(\rho)}(x)$ with respect to $\rho$, we obtain that $\hat{p}^{(\rho)}_1(x)>0$ and $\hat{p}^{(\rho)}_4(x)>0$
in $\bar{\Omega}$ for sufficiently small $\rho>0$.

To see that $\hat{p}^{(\rho)}_2(x)>0$ and $\hat{p}^{(\rho)}_3(x)>0$ in $\bar{\Omega}$ for sufficiently small $\rho>0$, we consider
\begin{equation}\label{4.2}
	u(x)=(u_1(x),u_2(x),u_3(x)):=
	\left(\frac{\partial \hat{p}^{(\rho)}_1}{\partial \rho}(x), \frac{\partial \hat{p}^{(\rho)}_2}{\partial \rho}(x),
	\frac{\partial \hat{p}^{(\rho)}_3}{\partial \rho}(x)\right)  \biggl|_{\rho=0}\,.
\end{equation}
Differentiating the equilibrium problem that $\hat{p}^{(\rho)}$ satisfies with respect to $\rho$ and then substituting
$\rho=0$, we obtain
\begin{subequations}\label{4.3}
\begin{align}
	&\De u_1+\la h_{14}(1-2\theta_{14})u_1 -\la\theta_{14}[h_{24}u_2+h_{34}u_3]-\theta_{14}(1-\theta_{14})=0
		&\quad&\text{in } \Omega\,, \label{4.3a}\\
	&\De u_2+\la (h_{24}-h_{14}\theta_{14})u_{2} +\theta_{14}(1-\theta_{14})=0
		&\quad&\text{in } \Omega\,, \label{4.3b}\\
	&\De u_3+\la (h_{34}-h_{14}\theta_{14})u_{3} +\theta_{14}(1-\theta_{14})=0
			&\quad&\text{in } \Omega\,, \label{4.3c} \\
	&\partial_\nu u_i=0\,, \quad i=1,2,3, &\quad&\text{on } \partial\Omega\,. \label{4.3d}
\end{align}
\end{subequations}

By our assumption that every eigenvalue $\mu$ of \eqref{4.1}
has positive real part, we infer from Lemma~\ref{lm4.1} that the smallest eigenvalue $\mu_{1}^{(2)}$ of \eqref{4.4} is positive.
By an inverse positivity result, from \eqref{4.3b} and the facts $\mu_{1}^{(2)}>0$ and $\theta_{14}(1-\theta_{14})>0$ we conclude that $u_2(x)>0$ in $\bar{\Omega}$.
(For the inverse positivity result, see e.g. Theorem 7.3 in \cite{Hess91}, in which we take
\begin{equation}\label{K}
	K=\left[-\De-\la (h_{24}-h_{14}\theta_{14})+c\right]^{-1}
\end{equation}
for some constant $c>0$ such that $-\la (h_{24}-h_{14}\theta_{14})+c>0$ in $\bar{\Omega}$, and associate it with zero Neumann boundary condition. Then
\begin{equation}\label{sprK}
	\mbox{spr}(K)=1/(\mu_{1}^{(2)}+c)\,,
\end{equation} 
and \eqref{4.3b} is equivalent to
\begin{equation}\label{4.7}
	\frac{1}{c} u_2-Ku_2 =  \frac{1}{c} K[\theta_{14}(1-\theta_{14})]\,.
\end{equation}
By standard elliptic regularity, embedding theory, and the strong maximum principle, $K$ is compact and strongly positive on $C^{1+\gamma}(\bar\Omega)$ for some $\gamma\in (0,1)$.  Moreover, $\mu_{1}^{(2)}>0$ and \eqref{sprK} imply that $1/c>\mbox{spr}(K)$, whence the positivity of the right-hand side of \eqref{4.7} leads to $u_2(x)>0$ in $\bar\Omega$.)

Similarly, we have
$\mu_{1}^{(3)}>0$ and $u_3(x)>0$ in $\bar{\Omega}$ as above.
Hence, we deduce from $u_i(x)>0$ in $\bar{\Omega}$ for
$i=2,3$ and \eqref{4.2} that $\hat{p}^{(\rho)}_2(x)>0$ and
$\hat{p}^{(\rho)}_3(x)>0$ for sufficiently small $\rho>0$. Thus, we have proved that
$\hat{p}^{(\rho)}$ is a full polymorphism for sufficiently small $\rho>0$, and this completes
the proof of (a).

Part (b) follows directly from Proposition \ref{prop:no_rec_general} and part (a).
\end{proof}

\begin{remark}\rm
Theorem \ref{thm:weak_reco} shows that a linearly stable equilibrium at either the 14-edge or the 23-edge moves into the interior of the state space if $\rh>0$. The following result shows that if such an equilibrium is unstable for $\rho=0 $, it leaves the state space when $\rho>0$.
\end{remark}

\begin{proposition}\label{prop:weak_reco}
 Suppose that for $\rho=0$ the edge equilibrium $\hat{p}^{(14)}$ $(\hat{p}^{(23)})$ exists and is nondegenerate  and linearly unstable.
Then there exists a neighbourhood in $\mathbf{X}$ of $\hat{p}^{(14)}$ $(\hat{p}^{(23)})$ in which there is no equilibrium of \eqref{dynamics_pi} for sufficiently small $\rho>0$. However, there is a family of stationary states $\hat{p}^{(\rho)}\notin\mathbf{X}$ such that $\hat{p}^{(\rho)}(x)\to \hat{p}^{(14)}(x)$ $(\hat{p}^{(23)}(x))$ uniformly as $\rho\to 0+$.
\end{proposition}

\begin{proof}
We prove this proposition for $\hat p^{(14)}$; the proof for $\hat p^{(23)}$ is analogous. Because we assume that $\hat p^{(14)}$ is nondegenerate, by the implicit
function theorem, there exists a {\it unique} family of equilibria $\hat{p}^{(\rho)}$ of \eqref{dynamics_pi} for $\rho>0$ sufficiently small such that
$\hat{p}^{(\rho)}(x)\to \hat{p}^{(14)}(x)$ uniformly as $\rho\to 0+$.

From Section~\ref{sec:existence_boundary} we know that $\hat p^{(14)}$ is always linearly stable with respect to \eqref{2.5}, and therefore the smallest eigenvalue $\mu_1^{(1)}$ of \eqref{4.5} is positive. Then Lemma~\ref{lm4.1} and the instability of $\hat p^{(14)}$ with respect to the full system \eqref{dynamics_pi} with $\rho=0$ imply that either $\mu_{1}^{(2)}<0$ or $\mu_{1}^{(3)}<0$.

If $\mu_{1}^{(2)}<0$, then by the same method we used in the proof of Theorem~\ref{thm:weak_reco}(a), we would have $1/c<\mbox{spr}(K)$ by \eqref{sprK}, whence the positivity of the right-hand side of \eqref{4.7} and \cite[Theorem 7.3]{Hess91} imply that $u_2$ cannot be a positive function in $\bar\Omega$. Thus, $\hat{p}^{(14)}$ leaves the state space when $\rho>0$. Similarly, if $\mu_{1}^{(3)}<0$, then $u_3$ cannot be a positive function in $\bar\Omega$, whence $\hat{p}^{(14)}$ again leaves the state space when $\rho>0$.

In light of the uniqueness of the family of $\hat{p}^{(\rho)}$ which converges to $\hat{p}^{(14)}$  as $\rho\to 0+$, we conclude that there exists a neighbourhood in $\mathbf{X}$ of $\hat{p}^{(14)}$ in which there is no equilibrium of \eqref{dynamics_pi} for sufficiently small $\rho>0$. This completes the proof.
\end{proof}

\section{Strong recombination}\label{sec:strong_reco}
Now we assume that recombination is sufficiently strong relative to
diffusion and selection, i.e., $r\gg 1$. We fix $d>0$ and $s>0$, hence
$\la>0$, work with \eqref{dynamics_pi}, and set $\ep=1/\rh>0$.
We study existence, uniqueness, and stability of two-locus clines for sufficiently small $\ep$ under the assumption (A). It will be convenient to follow the evolution of the allele frequencies $p_A=p_1+p_2$ and $p_B=p_1+p_3$, and the linkage disequilibrium $D=p_1p_4-p_2p_3$, instead of the gamete frequencies $p_i$. The corresponding transformation is given by
\begin{align}\label{trafo_T}
&\mathcal{T} : (p_A,p_B,D) \mapsto (p_1,p_2,p_3,p_4)\nonumber\\
	&\mathcal{T}(p_A,p_B,D) = (p_Ap_B+D,p_A(1-p_B)-D,(1-p_A)p_B-D,(1-p_A)(1-p_B)+D) \,.
\end{align}

It is easily shown that the system of differential equations \eqref{dynamics_pi_a} and \eqref{dynamics_pi_b} with the selection terms \eqref{S_i} is equivalent to
\begin{subequations}\label{eq:ABD_add}
\begin{align}
    \partial_t p_A &= \De p_A + \la\a(x) p_A(1-p_A) + \la\be(x)D  \,, \label{eq:ABD_a} \\
    \partial_t p_B &= \De p_B + \la\be(x) p_B(1-p_B) + \la\a(x)D   \,, \label{eq:ABD_b} \\
    \partial_t D  &= \De D + 2\nabla p_A\cdot\nabla p_B+ \la[\a(x)(1-2p_A) + \be(x)(1-2p_B)]D - \frac{1}{\ep} D  \label{eq:ABD_D} 	
\end{align}
in $\Om\times (0,\infty)$ and
\begin{equation}
	\partial_\nu p_A = \partial_\nu p_B = \partial_\nu D = 0 \quad\text{on } \partial\Om\times (0,\infty)\,.
\end{equation}
 \end{subequations}
Here, $\nabla$ denotes the vector differential operator with derivatives with respect to $x\in\Reals^n$. The constraints \eqref{constraint1} on the $p_i$ are transformed to
\begin{subequations}\label{constraints_twoloc}
\begin{equation}\label{constraint_pA_pB}
	0\le p_A\le1\,, \;  0\le p_B\le1\,,
\end{equation}
and
\begin{equation}\label{constraint_D}
	-\min\{p_A p_B,(1-p_A)(1-p_B)\} \le D \le \min\{p_A (1-p_B),(1-p_A)p_B\}\,,
\end{equation}
\end{subequations}
where these inequalities hold in $\Om\times[0,\infty)$ (e.g., \cite{RB2017}).
In particular, the map $\mathcal{T}:\mathbf{Y}\to\mathbf{X}$, given by \eqref{trafo_T}, is a homeomorphism, where
\begin{align}
	\mathbf{Y}:=\bigl\{&(v_1,v_2,v_3) \in  C(\bar\Omega; [0,1]^2) \times C(\bar\Omega;[-\tfrac14,\tfrac14]): \notag\\
		&-\min\{v_1 v_2, (1-v_1)(1-v_2)\} \leq v_3 \leq \min \{v_1(1-v_2), (1-v_1)v_2\}\bigr\}\,.
\end{align}
In addition, we define
\begin{equation}
	\mathbf{Y}_0 = \left\{(v_1,v_2,v_3)\in \mathbf{Y}: v_1\equiv 0 \text{ or } v_1\equiv 1 \text{ or } v_2\equiv 0 \text{ or } v_2\equiv 1\right\}
\end{equation}
and recall that each of the four edges in $\mathbf{Y}_0=\mathcal{T}^{-1}(\mathbf{X}_0)$ is invariant (Section \ref{sec:basic_dynamics}).

Because strong recombination erodes linkage disequilibrium rapidly, we expect that $D$ will be of order $\ep$ at stationarity (see \cite{RB2009,NHB1999} for related ODE models). If $D\equiv0$ then \eqref{eq:ABD_a} and \eqref{eq:ABD_b} describe two uncoupled one-locus systems, which are well understood (Section \ref{sec:SLClines}). We shall obtain the two-locus cline of \eqref{eq:ABD_add} as a perturbation of the Cartesian product of the two single-locus clines of \eqref{eq:ABD_a} and \eqref{eq:ABD_b} with $D\equiv0$. From Section~\ref{sec:SLClines}, and because we assume (A), we know that both exist if $\la>\max\{\la_{A},\la_{B}\}$, where $\la_{A}=\la_{\alpha}\in(0,\infty)$ and $\la_{B}=\la_{\beta}\in(0,\infty)$ are as in \eqref{lambda_h}.

For $h\in\{\alpha, \beta\}$, let $\theta_h(x)$ denote the global attractor of the single-locus problem at locus $\A$ or $\B$, respectively (Theorem~\ref{thm:singlelocus}). The following is the main result of this section.

\begin{theorem}\label{thm:7.1}
Suppose that (A) holds. For every $\la>0$ with $\la\neq\max \{ \la_A, \la_B\}$ and for sufficiently small $\ep>0$,
the system \eqref{eq:ABD_add} has an equilibrium $(\hat{p}_A, \hat{p}_B, \hat{D})= (\hat{p}_A^{(\ep)}, \hat{p}_B^{(\ep)}, \hat{D}^{(\ep)})$ that attracts all trajectories with initial data in $\mathbf{Y} \setminus \mathbf{Y}_0$, where convergence occurs in {$[C^2(\bar\Om)]^3$}. Moreover, the following conclusions hold.

\noindent {\rm (a)} For every $0 < \la <\max \{ \la_A, \la_B\}$, there exists $\epsilon_0>0$ such that
the system \eqref{eq:ABD_add} admits no internal equilibrium if $\epsilon \in (0,\epsilon_0]$. In fact, at least one of $\theta_\alpha$ and $\theta_\beta$ is trivial, and the globally attracting equilibrium
is independent of $\epsilon$, i.e.,
\begin{equation}
	(\hat{p}_A^{(\epsilon)}, \hat{p}_B^{(\epsilon)}, \hat{D}^{(\epsilon)}) = (\theta_\alpha, \theta_\beta, 0) \in \mathbf{Y}_0\,.
\end{equation}

\noindent {\rm (b)} For every $\la > \max \{ \la_A, \la_B\}$, there exists $\epsilon_0>0$ such that for every $\epsilon \in (0,\epsilon_0]$, the globally attracting equilibrium is internal and satisfies
\begin{equation}\label{eps_estimate_(pA,pB,D)}
		\| (\hat{p}_A^{(\ep)}, \hat{p}_B^{(\ep)}) - (\theta_\alpha, \theta_\beta)\|_{C^1(\bar\Omega)} + \|\hat{D}^{(\ep)}\|_{C(\bar\Omega)} = O(\epsilon)\,,
\end{equation}
i.e., $(\hat{p}_A^{(\ep)}, \hat{p}_B^{(\ep)}, \hat{D}^{(\ep)})$ lies in the interior of $\mathbf{Y}$ and converges to $(\theta_\alpha, \theta_\beta, 0)$ in {$C^1(\bar\Omega) \times C^1(\bar\Omega) \times C(\bar\Omega)$} as $\epsilon \to 0$.
\end{theorem}

\begin{remark}{\rm
By examining the elliptic system satisfied by the stationary solution $(\hat{p}_A, \hat{p}_B, \hat{D})$, and using the fact that $\|\hat D\|_{L^\infty(\Omega)} = O(\epsilon)$, it is not hard to show that $\|\hat D\|_{W^{2,p}(\Omega)} \leq C$ and thus $\|(\hat p_A, \hat p_B)\|_{C^{2,\gamma}(\bar\Omega)} \leq C$. This shows that in fact the convergence in \eqref{eps_estimate_(pA,pB,D)} can be improved to $C^2(\bar\Omega) \times C^2(\bar\Omega) \times C^1(\bar\Omega)$.}
\end{remark}

\begin{remark}\label{rem:7.2}{\rm
If $0<\la \le \min \{ \la_A, \la_B\}$, Theorem \ref{thm:7.1}(a) together with Theorem \ref{thm:singlelocus} implies that a monomorphic equilibrium is globally asymptotically stable for \eqref{eq:ABD_add} with sufficiently small $\epsilon>0$}.
\end{remark}

The case $\la = \max \{ \la_A, \la_B\}$ is degenerate and is briefly discussed in Section \ref{sec:Disc}. If (A) does not hold, then convergence to a boundary equilibrium occurs for every $\la>0$ (Section \ref{sec:Disc}).

\subsection{Preliminaries and proof of Theorem \ref{thm:7.1}(a)} Throughout this subsection, we assume that
$((p_A(x,0), p_B(x,0), D(x,0))\in \mathbf{Y}\setminus \mathbf{Y}_0$. Then, by Lemma~\ref{flow_into_interior}, the solution of \eqref{eq:ABD_add} satisfies $0<p_A(x,t)<1$ and $0<p_B(x,t)<1$ in $\bar\Omega\times (0,\infty)$.
For convenience, we define
\begin{equation}\label{wAB}
	\DA(x,t)= \dfrac{D(x,t)}{p_A(x,t)(1-p_A(x,t))}\,,\quad \DB(x,t)= \dfrac{D(x,t)}{p_B(x,t)(1-p_B(x,t))}\,.
\end{equation}

\begin{lemma}\label{lem:7.4a}
For given $\la>0$, there exists $C_0>0$ independent of $\epsilon$ such that
\begin{equation}\label{eq:lem7.4-1}
	\sup_{x \in \Omega, t \geq 1} \left[ \frac{|\nabla p_A(x,t)|}{p_A(x,t)(1-p_A(x,t))} +  \frac{|\nabla p_B(x,t)|}{p_B(x,t)(1-p_B(x,t))}
	\right] \leq C_0\,.
\end{equation}
In particular,
\begin{equation}\label{eq:lem7.4-2}
	| \nabla p_A(x,t)| + |\nabla p_B(x,t)|  \leq C_0 \quad \text{ for } (x,t) \in \Omega\times[1,\infty)\,.	
\end{equation}
\end{lemma}

\begin{proof}
In light of \eqref{wAB}, we can rewrite \eqref{eq:ABD_a} and its boundary condition as
\begin{equation}\label{eq:pA}
	\begin{cases}
	\partial_t p_A - \Delta p_A = \la \left[\alpha + \beta \DA \right](1-p_A)p_A   & \text{ in }\Om\times (0,\infty)\,, \\
	\partial_\nu p_A = 0 &\text{ on } \partial\Om\times (0,\infty)\,,
	\end{cases}
\end{equation}
where, by the constraints \eqref{constraints_twoloc},
\begin{equation}\label{eq:wA}
	|\DA |  = \frac{|D|}{p_A} + \frac{|D|}{1-p_A} \leq 2\,.
 \end{equation}
Hence, $p_A \geq 0$ satisfies the differential inequality
\begin{equation}
\begin{cases}
	\partial_t p_A - \Delta p_A \leq M_0 p_A  &\text{ in } \Omega \times (0,\infty)\,,\\
	\partial_\nu p_A = 0 &\text{ on }  \partial\Omega \times(0,\infty)\,,
\end{cases}
\end{equation}
where $M_0 = \la (\|\alpha\|_{C(\bar\Omega)} + 2\|\beta\|_{C(\bar\Omega)})$. By comparison we obtain
\begin{equation*}
	\|p_A\|_{C(\bar\Omega \times [t-1, t])} \leq e^{M_0} \|p_A(\cdot, t-1)\|_{C(\bar\Omega)} \quad \text{ for }t \geq 1.
\end{equation*}
 Now, we may apply a parabolic $L^p$-estimate to the solution $p_A$ of \eqref{eq:pA} and obtain a constant $C_1>0$ (independent of $t \geq 1$) such that
$$
	\|p_A(\cdot,t)\|_{C^1(\bar\Omega)} \leq C_1 \|p_A(\cdot,t-1)\|_{C(\bar\Omega)} \quad \text{ for }t \geq 1.
$$
Hence,
\begin{equation}\label{eq:wAA}
	\sup_{x \in \Omega} \frac{\left|\nabla p_A(x,t)\right|}{p_A(x,t)} \leq  \frac{C_1\sup_{x' \in \Omega} p_A(x', t-1)}{\inf_{x' \in \Omega} p_A(x',t)} \leq C_2 \quad \text{ for }t \geq 1,	
\end{equation}
where the second inequality is based on a standard Harnack inequality for homogeneous parabolic equations with uniformly bounded coefficients \cite[Corollary 7.42]{lieberman}. (Due to the Neumann boundary condition, the Harnack inequality can be applied up to the boundary of the spatial domain $\Omega$.)

By repeating the argument with $1-p_A$, we obtain
\begin{equation}\label{eq:wAAA}
	\sup_{x \in \Omega} \frac{|\nabla p_A(x,t)|}{1-p_A(x,t)}  = \sup_{x \in \Omega} \frac{|\nabla (1-p_A(x,t))|}{1-p_A(x,t)}\leq  C_3 \quad \text{ for }t \geq 1.
\end{equation}
Combining \eqref{eq:wAA} and \eqref{eq:wAAA}, we deduce
\begin{equation}
	\sup_{x \in \Omega} \frac{|\nabla p_A|}{p_A(1-p_A)}
= \sup_{x \in \Omega}\left[ \frac{|\nabla p_A|}{p_A} + \frac{|\nabla p_A|}{1-p_A} \right]
 \leq C_2+C_3 \quad \text{ for } t\geq 1\,.
\end{equation}
The corresponding estimate for $p_B$ follows analogously.
\end{proof}

\begin{remark}\label{remark:lem7.4}
{\rm The parabolic $L^p$ estimate and the Harnack inequality require only the boundedness of $\la$. Therefore, for each fixed $M>0$, the bound $C_0$ in Lemma  \ref{lem:7.4a} can be chosen uniformly for $\la \in (0,M]$ and $\epsilon \in (0,\infty)$.}
\end{remark}

\begin{lemma}\label{lem:7.4}
For given $\la>0$ and $\epsilon>0$ such that
\begin{equation}\label{eq:epsilon_smaller}
\frac{1}{2\epsilon} > 3 \la \|\beta \|_{C(\bar\Omega)} + 2 C_0^2,
\end{equation}
where $C_0$ is as in Lemma \ref{lem:7.4a}, we have
\begin{subequations}\label{eq:cor7.4a}
\begin{alignat}{2}
	&\limsup_{ t\to\infty} \left\|\DA(\cdot, t)\right\|_{C(\bar\Omega)}  \leq  4 \epsilon C_0 \limsup_{t \to \infty} \|\nabla p_B\|_{C(\bar\Omega)}  \,, \label{eq:cor7.4aa}\\
	&\limsup_{ t\to\infty} \left\|\DB(\cdot, t)\right\|_{C(\bar\Omega)}  \leq  4 \epsilon C_0 \limsup_{t \to \infty} \|\nabla p_A\|_{C(\bar\Omega)}  \,.\label{eq:cor7.4ab}
\end{alignat}
\end{subequations}
In particular, the following holds:\\
\noindent {\rm (a)}
if $p_B(\cdot,t) \to 0$ or $1$ in $C(\bar\Omega)$, then $\DA(\cdot,t) \to 0$ in $C(\bar\Omega)$;\\
\noindent {\rm (b)} if $p_A(\cdot,t) \to 0$ or $1$ in $C(\bar\Omega)$, then $\DB(\cdot,t) \to 0$ in $C(\bar\Omega)$;\\
\noindent {\rm (c)}
$\limsup_{ t\to\infty} \left\|\DA(\cdot, t)\right\|_{C(\bar\Omega)}  \leq  4 \epsilon C_0^2$ \;and\;
	$\limsup_{ t\to\infty} \left\|\DB(\cdot, t)\right\|_{C(\bar\Omega)}  \leq  4 \epsilon C_0^2$.
\end{lemma}

\begin{remark}\label{rem:D_to_O(eps)}\rm
It is easy to deduce from \eqref{wAB} and Lemma \ref{lem:7.4}(c) that for every $\la>0$,
\begin{equation}\label{eq:cor7.4b}
	\limsup_{t \to \infty} \|D(\cdot,t)\|_{C(\bar\Omega)}  \leq \limsup_{t \to \infty} \left\|\DA(\cdot, t) \right\|_{C(\bar\Omega)} = O(\epsilon)
\end{equation}
as $\ep\to0$.
This shows that indeed, as argued verbally in Section \ref{sec:reco_LD} and above, linkage disequilibrium decays to values close to 0 if recombination is sufficiently strong. Similar results were proved previously for general non-spatial multilocus models \cite{N1993,NHB1999} as well as for spatial models with a finite number of demes \cite{RB2009}. However, \eqref{eq:cor7.4a} is stronger than \eqref{eq:cor7.4b}, and it will be essential for the proof of Theorem \ref{thm:7.1}.
\end{remark}

\begin{proof}[Proof of Lemma \ref{lem:7.4}]
From \eqref{eq:ABD_a}, \eqref{eq:ABD_D}, and \eqref{wAB}, we derive
\begin{subequations}\label{eq:wAeq}
\begin{alignat}{2}
	&\partial_t \DA  - \Delta \DA  -\frac{2(1-2p_A)\nabla p_A}{p_A(1-p_A)}\cdot \nabla \DA
	+ \DA \Biggl[\la \beta(1-2p_A)\DA  - \la \beta (1-2p_B) \notag \\
	&\qquad+ \frac{2|\nabla p_A|^2}{p_A(1-p_A)} + \frac{1}{\epsilon} \Biggr] = \frac{2\,\nabla p_A \cdot \nabla p_B }{p_A(1-p_A)}   &&\hspace{-30mm}\text{in } \Omega\times(0,\infty)\,,\\
		&\partial_\nu \DA  = 0 &&\hspace{-30mm}\text{on } \partial\Omega\times(0,\infty)\,.
\end{alignat}
\end{subequations}
Because each of $\underlineDA   \in \{\DA,-\DA \}$ satisfies the differential inequality
\begin{align*}
	& \partial_t \underlineDA   - \Delta \underlineDA    -\frac{2(1-2p_A)\nabla p_A}{p_A(1-p_A)}  \cdot \nabla  \underlineDA    \\
	&\quad+\underlineDA  \left[\la \beta(1-2p_A)\DA  - \la \beta (1-2p_B)  + \frac{2|\nabla p_A|^2}{p_A(1-p_A)} + \frac{1}{\epsilon} \right]
	\leq  \frac{2|\nabla p_A | }{p_A(1-p_A)} \left|\nabla p_B\right|
 \end{align*}
for a subsolution, their maximum $|\DA | = \max\{\DA , -\DA \}$ satisfies the same differential inequality in the weak sense.

From \eqref{eq:epsilon_smaller}, \eqref{eq:lem7.4-1}, and \eqref{eq:lem7.4-2}, we obtain
\begin{align*}
	\frac{1}{2\epsilon} &\geq \sup_{x \in \Omega,\, t \geq 1} \left(3 \la |\beta(x)| + \frac{2|\nabla p_A|}{p_A(1-p_A)}|\nabla p_A|\right) \notag \\
	&\geq \sup_{x \in \Omega,\, t \geq 1} \left| \la \beta (1-2p_A)\DA  - \la \beta (1-2p_B) +  \frac{2|\nabla p_A|^2}{p_A(1-p_A)}\right|\,,
\end{align*}
where we used the fact $|\DA |\leq 2$ by \eqref{eq:wA}. Then $|\DA |$ is a weak subsolution of
\begin{subequations}\label{eq:abswAeq}
\begin{alignat}{2}
	&\partial_t \mathcal{D} - \Delta \mathcal{D} - \frac{2(1-2p_A)\nabla p_A}{p_A(1-p_A)}  \cdot  \nabla\mathcal{D} +\frac{\mathcal{D}}{2\epsilon}  = \frac{2|\nabla p_A | }{p_A(1-p_A)} \left|\nabla p_B\right|
	&&\quad\text{ in } \Omega\times(0,\infty)\,,\\
	&\partial_\nu \mathcal{D} = 0 &&\quad\text{ on } \partial\Omega\times(0,\infty)\,.
\end{alignat}
\end{subequations}

Now for every $t_0 \geq 1$, we may construct a supersolution of \eqref{eq:abswAeq} in the domain $\Omega \times [t_0,\infty)$ as follows:
$$
\overlineDA  := 4\epsilon \sup_{t' \geq t_0} \left[\left\| \frac{\nabla p_A(\cdot,t')}{p_A(\cdot,t')(1-p_A(\cdot,t'))}\right\|_{C(\bar\Omega)} \left\| \nabla p_B(\cdot, t')\right\|_{C(\bar\Omega)}\right] +2 e^{-(t-t_0)/(2\epsilon)}.
$$
Then, clearly, $\overlineDA   \geq 2 \geq |\DA |$ for $x\in \Omega$ and $t = t_0$. Hence, we can deduce by comparison that
\begin{equation}\label{eq:est_wA}
		\sup_{x\in\Omega}|\DA(x,t)| \leq  4\epsilon \sup_{t' \geq t_0} \left[\left\| \frac{\nabla p_A(\cdot,t')}{p_A(\cdot,t')(1-p_A(\cdot,t'))}\right\|_{C(\bar\Omega)} \left\| \nabla p_B(\cdot, t')\right\|_{C(\bar\Omega)}\right] +2 e^{-(t-t_0)/(2\epsilon)}
\end{equation}
for $t \geq t_0$. By letting $t \to \infty$ and then $t_0 \to \infty$, we obtain \eqref{eq:cor7.4aa}. An analogous argument for $\DB $ yields \eqref{eq:cor7.4ab}.

For assertion (a), we observe that if $p_B(\cdot,t)$ approaches $0$ or $1$ uniformly as $t \to \infty$, then Lemma \ref{lem:7.4a} informs us that $\|\nabla p_B(\cdot,t)\|_{C(\bar\Omega)} \to 0$ as $t \to \infty$. Hence,
we obtain assertion (a) by \eqref{eq:cor7.4aa}. The proof of (b) is analogous and is omitted. Part (c) follows directly from \eqref{eq:cor7.4a} and \eqref{eq:lem7.4-2}.
\end{proof}

\begin{remark}\label{remark:lem7.5} {\rm From Remark \ref{remark:lem7.4} and \eqref{eq:epsilon_smaller}, we {conclude} that for each $M>0$,  the estimates in  \eqref{eq:cor7.4a}, Lemma \ref{lem:7.4}(c), and \eqref{eq:cor7.4b} hold for $C_0$ chosen uniformly for $\la \in (0,M]$ and for $\epsilon  <(6\|\beta\|_{C(\bar\Omega)} + 4 C_0^2)^{-1}$}.
\end{remark}

\begin{lemma}\label{lem:7.4c}
 {\rm (a)} If $0<\la < \la_A$, there exists $\tilde\epsilon_a>0$ such that for $\epsilon \in (0,\tilde\epsilon_a]$,
$$
	\lim_{ t\to\infty} p_A(\cdot,t)
	= \left\{ \begin{array}{ll} 1 &\text{ if }\bar\alpha >0,\\
	0 &\text{ if }\bar\alpha<0 \end{array}\right.  \quad \text{ in }C^1(\overline\Omega)\,.
$$
 \noindent {\rm (b)} If $0<\la < \la_B$, there exists $\tilde\epsilon_b>0$ such that for $\epsilon \in (0,\tilde\epsilon_b]$,
$$
	\lim_{ t\to\infty} p_B(\cdot,t) = \left\{ \begin{array}{ll} 1 &\text{ if }\bar\beta >0\,,\\
	0 &\text{ if }\bar\beta<0 \end{array}\right.  \quad \text{ in }C^1(\overline\Omega).
$$
 \noindent {\rm (c)}  If $0<\la < \max\{\la_A, \la_B\}$, there exists $\tilde\epsilon = \max\{\tilde\epsilon_a, \tilde\epsilon_b\}>0$ such that for $\epsilon \in (0,\tilde\epsilon]$,
$$
	\lim_{t \to \infty} D(\cdot,t) = 0 \;\text{ in } C(\bar\Omega)\,.
$$
\end{lemma}

\begin{proof}
First, we prove (a) and suppose $\bar\alpha<0$. By Theorem \ref{thm:singlelocus}, $0$ is a linearly stable equilibrium of
\begin{subequations}\label{eq:theta_time}
\begin{alignat}{2}
		&\partial_t \theta - \Delta \theta = \la \alpha \theta(1-\theta) &&\quad\text{ in } \Omega\times(0,\infty)\,, \\
		&\partial_\nu \theta  =0 &&\quad\text{ on } \partial\Omega\times(0,\infty)\,,
\end{alignat}
\end{subequations}
and it attracts all solutions of \eqref{eq:theta_time} that are not identically equal to $1$. Because $\alpha$ changes sign and $\bar\alpha<0$, for $\delta_1>0$ sufficiently small, $\alpha+\delta_1$ still changes sign and $\overline{\alpha+\delta_1}<0$.  Moreover, $\la^{*}(\alpha+\delta_1)$, defined below \eqref{1.6}, decreases continuously as $\delta_1$ increases from $0$ [46, Proposition~1.5]. Because $\la<\la_A=\la^*(\alpha)$, we may choose $\delta_1$ sufficiently small such that $\la<\la^{*}(\alpha+\delta_1)$. Therefore, $0$ is globally asymptotically stable also for
\begin{subequations}\label{eq:theta_time_upper}
\begin{alignat}{2}
		&\partial_t \overline\theta - \Delta \overline\theta = \la (\alpha+ \delta_1)\overline\theta(1  -\overline\theta) &&\quad\text{ in } \Omega\times(0,\infty)\,, \\
		&\partial_\nu \overline\theta  =0 &&\quad\text{ on } \partial\Omega\times(0,\infty)\, .
\end{alignat}
\end{subequations}

By Lemma \ref{lem:7.4}(c), let $\epsilon$ be sufficiently small so that for some $t_0>0$,  $\left|\beta \DA \right| \leq \delta_1$ in $\Omega\times[t_0,\infty)$, and
let $\overline\theta$ be a solution of \eqref{eq:theta_time_upper} with initial condition $\overline\theta(x,t_0) = p_A(x,t_0)$. Then
$$
\partial_t p_A - \Delta p_A = \la \left[\alpha + \beta \DA \right] p_A(1-p_A) \leq \la(\alpha+ \delta_1) p_A(1-p_A)
$$
on $\Omega\times[t_0,\infty)$. Since also $\partial_\nu p_A = 0$ on $\partial\Omega \times(0,\infty)$ and $p_A(x,t_0)  = \overline\theta(x,t_0)$ in $\Omega$, we deduce by comparison that
$$
	0 \leq p_A(x,t) \leq \overline\theta(x,t) \quad \text{ in } \Omega\times[t_0,\infty)\,.
$$
Because $\left\|\overline\theta(\cdot,t)\right\|_{C(\bar\Omega)} \to 0$ as $t \to \infty$, we have $\left\|p_A(\cdot,t)\right\|_{C(\bar\Omega)} \to 0$ as $t \to \infty$. By parabolic regularity, we obtain $\|p_A(\cdot,t)\|_{C^1(\bar\Omega)} \to 0$.
This proves (a) if $\bar\alpha<0$. The proofs of (a) for $\bar\alpha >0$ and of (b) are analogous and are omitted.

By \eqref{constraints_twoloc}, statement (c) follows directly from (a) and (b).
\end{proof}

\begin{remark}\label{remark:lem7.5b}
{\rm For every given $\delta>0$, the constant $\delta_1$ in the above proof can be chosen uniformly for $\la \in (0, \la_A - \delta]$.
Hence by Lemma \ref{lem:7.4} and Remark \ref{remark:lem7.5} one can choose
$\tilde\epsilon_a$ (resp., $\tilde\epsilon_b$) uniformly for $\la \in (0,\la_A - \delta]$ (resp. $\la \in (0,\la_B - \delta]$).  }
\end{remark}

\begin{lemma}\label{lem:U1}
Suppose $\la > \la_A$, and define
\begin{equation}\label{def_L_ph}
	L_\ph = -\Delta - \la\alpha (1-2\ph)\,.
\end{equation}
Then there exists $\delta_1>0$ such that if
$\ph \in C(\bar\Omega)$ satisfies $\|\ph - \theta_\alpha\|_{C(\bar\Omega)} < \delta_1$, then
\begin{equation}\label{eq:spectrum}
	\sigma(L_\ph) \subset \{ z \in \mathbb{C}: \textup{Re}\, z > \delta_0\} \quad \text{ for some }\delta_0 >0.
\end{equation}
\end{lemma}

\begin{proof}
Because $\la > \la_A$, the positive equilibrium $\theta_\alpha$ is linearly stable in the single-locus problem, i.e., there exists $\delta_0>0$ such that the operator $L_{\theta_\alpha}$
satisfies $\sigma(L_{\theta_\alpha}) \subset \{z \in \mathbb{C}: \textup{Re}\, z \geq 2\delta_0\}$. The lemma thus follows from upper semicontinuity of the spectrum of $L_\ph$ with respect to the coefficient $\ph \in C(\bar\Omega)$.
\end{proof}

\begin{lemma}\label{lem:U2}
Suppose $q_A$ is a solution of
\begin{subequations}
\begin{alignat}{2}
	&\partial_t q_A + L_\ph q_A  = F(x,t) &&\quad\text{ in } \Omega\times(t_0,\infty)\,, \\
	&\partial_\nu q_A = 0 &&\quad\text{ on } \partial\Omega \times (t_0,\infty) \,,
\end{alignat}
\end{subequations}
where $L_\ph$ satisfies \eqref{eq:spectrum} and   $F(x,t) \in C(\bar\Omega \times [t_0, \infty))$. Then there exists $C'>0$ (which depends on $L_\varphi$ but is independent of $F$) such that
\begin{equation}\label{tilde_qA:lemU2}
	\limsup_{t \to \infty} \|q_A(\cdot,t)\|_{C^1(\bar\Omega)} \leq C'\limsup_{t\to\infty} \|F(\cdot,t)\|_{C(\bar\Omega)}\,.
\end{equation}
\end{lemma}

\begin{proof}
By the variation-of-constants formula, we have
\begin{equation}\label{eq:voc}
	q_A(\cdot,t) = e^{-(t-t_0)L_\ph} q_A(\cdot,t_0) + \int_{t_0}^t e^{-(t-s) L_\ph} F(\cdot,s)\,ds  \quad \text{ for } t > t_0\,,
\end{equation}
where $e^{-t L_\ph}$ is the semigroup generated by $L_\ph$ under homogeneous Neumann boundary conditions. Using \eqref{eq:spectrum},  it is a consequence of \cite[(2.3.3)]{Lunardi} that for every $\ga\in(0,1)$ and $p\ge1$ there is a constant $c>0$ such that
\begin{equation*}
	 \|e^{-t L_\ph} w\|_{D_{L_\ph}(\ga,\infty)} \leq c t^{-\ga} e^{-\delta_0 t} \| w\|_{C(\bar\Omega)}   \quad \text{for all }t > 0\,,
\end{equation*}
where $D_{L_\varphi}(\gamma,\infty)$ is the real interpolation space between $C(\bar\Omega)$ and the domain $D(L_\ph)=\bigcap_{p\ge1} W^{2,p}(\Omega)$.
Because $D_{L_\varphi}(\gamma,\infty)\subseteq C^{1,2\gamma-1}(\bar\Omega)$ if $\gamma\in(\tfrac12,1)$ \cite[Theorem 3.1.30]{Lunardi}, we obtain
\begin{equation}\label{estimate_etLph}
	\|e^{-t L_\ph} w\|_{C^{1}(\bar\Omega)} \leq c t^{-\ga} e^{-\delta_0 t} \| w\|_{C(\bar\Omega)}   \quad \text{for all }t > 0\,.
\end{equation}
Applying \eqref{estimate_etLph} to \eqref{eq:voc}, we derive
$$
	\|q_A(\cdot,t)\|_{C^1(\bar\Omega)}
	\leq c(t-t_0)^{-\ga}e^{-\delta_0(t-t_0)}\| q_A(\cdot, t_0)\|_{C(\bar\Omega)}
		+ \int_{t_0}^t c (t-s)^{-\ga} e^{-\delta_0(t-s)}\|F(\cdot,s)\|_{C(\bar\Omega)}\, ds
$$
for $t > t_0 > 0$. Letting $t \to \infty$, we arrive at \eqref{tilde_qA:lemU2}.
\end{proof}

\begin{proposition}\label{prop:persistence}
\noindent {\rm (a)} If $\la>\la_A$, then for every trajectory $(p_A, p_B, D)$ of \eqref{eq:ABD_add} with initial data $p_A(\cdot,0) \notin\{0, 1\}$, we have
\begin{equation}\label{eq:prop2a}
	\limsup_{t \to\infty}  \left\|p_A(\cdot,t) - \theta_\alpha \right\|_{C^1(\bar\Omega)} = O(\epsilon) \quad \text{ as }\epsilon \to 0.
\end{equation}

\noindent {\rm (b)} If $\la >\la_B$, then for every trajectory $(p_A, p_B, D)$ of \eqref{eq:ABD_add}  with initial data $p_B(\cdot,0) \notin\{0, 1\}$, we have
\begin{equation}\label{eq:prop2b}
	\limsup_{t \to\infty}  \left\|p_B(\cdot,t) - \theta_\beta \right\|_{C^1(\bar\Omega)} = O(\epsilon)\quad \text{ as }\epsilon \to 0.
\end{equation}
\end{proposition}

\begin{proof}
To prove (a), assume $\la>\la_A$. We may choose a constant $\delta_2>0$ sufficiently small such that
$$
	\la > \la_{\alpha + \delta}  \quad \text{ for all $\delta$ with } |\delta| \leq \delta_2\,,
$$
where $\la_{\alpha + \delta}$ is defined in \eqref{lambda_h}.
Then the single-locus equation
\begin{subequations}\label{eq:theta_time_upper2a}
\begin{alignat}{2}
		&\partial_t \theta - \Delta \theta = \la (\alpha+ \delta)\theta(1  -\theta) &&\quad\text{ in } \Omega\times(0,\infty)\,,\\
		&\partial_\nu \theta  =0 &&\quad\text{ on } \partial\Omega\times(0,\infty)
\end{alignat}
\end{subequations}
has a unique globally asymptotically stable equilibrium $\theta_{\alpha+\delta}$. 
Let $C'$ be given by Lemma \ref{lem:U2}, where $L_\varphi = L_{\theta_\alpha}$, and let
 $\delta' = \frac{1}{C' \|\la\alpha\|_{C(\bar\Omega)}}$. We claim that for each sufficiently small $\epsilon$, and any non-trivial initial condition,
\begin{equation}\label{eq:est_c0}
	\limsup_{t \to \infty} \|p_A(\cdot,t) - \theta_\alpha\|_{C(\bar\Omega)} < \delta'. 
\end{equation}

To prove \eqref{eq:est_c0}, let $\delta'>0$ be given as above. Since $\la >\la_A$, the  steady state $\theta_{\alpha+\delta}$ depends continuously on $\delta \in [-\delta_2,\delta_2]$, thus there exists $\eta  \in (0,\delta_2)$ (depending on $\delta'$) such that 
\begin{equation}\label{eq:iftift}
	\|\theta_{\alpha - \eta} - \theta_{\alpha}\|_{C(\bar\Omega)} + 		\|\theta_{\alpha + \eta} - \theta_{\alpha}\|_{C(\bar\Omega)} < \delta'.
\end{equation}
Next, fix $\epsilon>0$ small enough so that $4\epsilon C_0^2\|\beta\|_{C(\bar\Omega)} < \eta$ and \eqref{eq:epsilon_smaller} are satisfied (where $C_0$ is as in Lemma \ref{lem:7.4a}). Then,
by Lemma \ref{lem:7.4}, there exists $t_0>0$ such that 
$$
	\left|\frac{\beta D}{p_A(1-p_A)}\right| = \left| \beta D_A\right| \leq  \eta \quad\text{ in } \Omega\times[t_0,\infty)\,.
$$
In this case, $p_A$ satisfies
$$
	\la (\alpha - \eta)p_A(1-p_A) \le \partial_t p_A - \Delta p_A
\leq \la (\alpha + \eta)p_A(1-p_A) \quad\text{ in } \Omega\times[t_0,\infty) \,.
$$
Hence, by comparison and by the fact that $\theta_{\alpha\pm \eta}$ is the globally asymptotically stable equilibrium of \eqref{eq:theta_time_upper2a} with $\delta = \pm \eta$, respectively, we deduce that
$$
	\theta_{\alpha - \eta}(x) \leq \liminf_{t \to \infty} p_A(x,t) \leq \limsup_{t \to \infty} p_A(x,t) \leq \theta_{\alpha+ \eta}(x).
$$
Combining this with \eqref{eq:iftift}, we obtain
\begin{align*}
&\quad \limsup_{t \to \infty}\|p_A(\cdot,t)-\theta_\alpha\|_{C(\bar\Omega)} \\
 &\qquad\leq  \limsup_{t\to\infty} \left[\max_{x\in\bar\Omega}  (p_A(x,t) - \theta_\alpha(x))_+  \right]+  \limsup_{t\to\infty} \left[\max_{x\in\bar\Omega} (\theta_\alpha (x)- p_A(x,t))_+ \right] \\
&\qquad\leq\| \theta_{\alpha + \eta} - \theta_\alpha\|_{C(\bar\Omega)} +  \| \theta_{\alpha } - \theta_{\alpha - \eta}\|_{C(\bar\Omega)} < \delta'\,,
\end{align*}
which proves \eqref{eq:est_c0}.

Next, let $q_A(x,t) = p_A(x,t) - \theta_\alpha(x)$ and
$F(x,t)= -\la \alpha  (q_A)^2 + \la \beta \DA  p_A(1-p_A)$.
Then
\begin{equation}
	\partial_t q_A + L_{\theta_\alpha} q_A = F(x,t)\,,
\end{equation}
where
$L_{\theta_\alpha}$ is defined according to \eqref{def_L_ph}. Since $\la >\la_A$, the equilibrium $\theta_\alpha$ is linearly stable and thus $\sigma(L_{\theta_\alpha}) \subset \{z \in \mathbb{C}: \textup{Re}\, z > \delta_0\}$ for some $\delta_0>0$. Because
\eqref{eq:cor7.4b} entails $\limsup_{t\to\infty}\|F(\cdot,t)\|_{C(\bar\Omega)}= O(\epsilon) + \|\la \alpha\|_{C(\bar\Omega)} \limsup_{t\to\infty}\|q_A(\cdot,t)\|_{C(\bar\Omega)}^2$, we can invoke Lemma \ref{lem:U2} to deduce that for some constant $C'>0$ (the same as the one at the beginning of the proof), we have
\begin{equation}\label{eq:implies}
	\limsup_{t \to \infty} \|q_A(\cdot,t)\|_{C^1(\bar\Omega)} = C'\left[O(\epsilon) + \|\la \alpha\|_{C(\bar\Omega)}  \limsup_{t \to \infty} \|q_A(\cdot,t)\|^2_{C(\bar\Omega)} \right].
\end{equation}
By our choice of $\delta'=\frac{1}{2C' \|\la \alpha\|_{C(\bar\Omega)}}$ and  \eqref{eq:est_c0}, we have $$	C' \|\la \alpha\|_{C(\bar\Omega)} \limsup_{t \to \infty} \|q_A(\cdot,t)\|^2_{C(\bar\Omega)} \leq \frac{1}{2} \limsup_{t \to \infty} \|q_A(\cdot,t)\|_{C(\bar\Omega)}, $$
and  \eqref{eq:implies} yields
$$
\limsup_{t \to \infty} \|q_A(\cdot,t)\|_{C^1(\bar\Omega)} = O(\epsilon).
$$ 

This proves (a) for $\la>\la_A$.
The proof of (b) is analogous.
\end{proof}

We end this subsection with the proof of Theorem \ref{thm:7.1}(a).
\begin{proof}[Proof of Theorem \ref{thm:7.1}(a)]
Let $\la < \max\{\la_A,\la_B\}$. Without loss of generality, we assume $\la < \la_B$.
Then by Lemma \ref{lem:7.4c}(b) and Lemma \ref{lem:7.4}(a) we have
\begin{equation}\label{eq:LN2002}
	p_B(\cdot,t) \xrightarrow{t\to\infty} 0 \text{ or } 1 \;\text{ in } C^1(\bar\Om) \quad \text{and}\quad \DA \xrightarrow{t\to\infty} 0 \;\text{ in }C(\bar\Omega)\,,
\end{equation}
respectively.
Hence, equation \eqref{eq:ABD_a} for $p_A$  is asymptotic to \eqref{eq:theta} with $h=\alpha$.

Now, for \eqref{eq:theta} with $h = \alpha$, the equilibrium $\theta_\alpha$ is globally asymptotically stable (recall that $0<\theta_\alpha<1$ if $\la>\la_A$, and $\theta_\alpha \in\{0,1\}$ if $\la\le\la_A$). Any other equilibrium in $\{0,1\}$ is linearly unstable. For every given trajectory $\{p_A(\cdot,t)\}_{t \geq 0}$ of \eqref{eq:ABD_a}, the omega limit set $\omega_0$ is an internally chain-transitive set of the semiflow generated by the limiting equation \eqref{eq:theta} with $h=\alpha$. In particular, $\omega_0$ must be a singleton set containing one of the equilibria $\{0,\theta_\alpha,1\}$,  i.e., $p_A(\cdot,t)$ converges to one of the equilibria as $t\to\infty$.

To prove that $p_A(\cdot,t) \to \theta_\alpha$ in $C^1(\bar\Om)$, we consider the case $\bar{\alpha}>0$ first. If $0<\la<\la_A$, then $\theta_\alpha=1$ and $p_A(\cdot,t) \to 1$ follows from Lemma\,\ref{lem:7.4c}(a). If $\la>\la_A$, then $0<\theta_\alpha(x)<1$ on $\bar{\Omega}$
and Proposition\,\ref{prop:persistence}(a) excludes the possibility that $p_A(\cdot,t)\to 0$ or $1$ and thus leads to $p_A(\cdot,t) \to \theta_\alpha$. If $\la=\la_A$, then $\theta_\alpha=1$, and $0$ is linearly unstable as an equilibrium of \eqref{eq:theta} with $h=\alpha$. We rewrite the equation (\ref{eq:ABD_a}) for $p_A$ as
\begin{subequations}\label{eq:LN2002-2}
\begin{alignat}{2}
		&\partial_t p_A - \Delta p_A = \la\alpha(1-p_A)p_A + \la g(x,t) p_A &&\quad\text{ in } \Omega\times (0,\infty) \,,\\
		&\partial_\nu p_A = 0 	&&\quad\text{ on } \partial\Omega \times (0,\infty)\,,\\
		&p_A(x,0) \geq 0 \,\text{ and }\, p_A(x,0)\not\equiv 0 	&&\quad\text{ in } \bar\Omega\,,
\end{alignat}
\end{subequations}
where, by \eqref{eq:LN2002},
\begin{equation}
	g(x,t) = \beta(x)\DA(1-p_A(x,t)) \to 0 \text{ in } C(\bar\Omega) \text{ as } t\to \infty\,.
\end{equation}
Thus, we may apply \cite[Lemma 2.5]{LN2002} to deduce that $p_A(\cdot,t)\to 1$ in $C^1(\bar\Om)$ as $t\to\infty$. For each fixed $\epsilon$, the convergence of $(p_A, p_B, D)$ as $t \to\infty$ can in fact be improved to $[C^2(\bar\Omega)]^3$, via parabolic regularity.
This completes the proof of $p_A(\cdot,t) \to \theta_\alpha$ as $t\to\infty$.
Finally,  the proof for the case $\bar{\alpha}\le 0$ is similar and is omitted.
\end{proof}

\begin{remark}\rm 
Here is an alternative proof of Theorem \ref{thm:7.1}(a) without using the chain transitivity. As above, we consider the case $\bar{\alpha}>0$. If $0<\la\le\la_A$, we apply \cite[Lemma 2.5]{LN2002} to equation \eqref{eq:LN2002-2} and conclude that $p_A(\cdot,t)\to 1$ as $t\to\infty$. If $\la>\la_A$, we apply \cite[Lemma 2.5]{LN2002} to both $p_A$ and $(1-p_A)$ to obtain
\begin{equation*}
\liminf_{t\to\infty}p_A(x,t)\ge \theta_\alpha(x)
\end{equation*}
and
\begin{equation*}
\liminf_{t\to\infty}(1-p_A(x,t))\ge 1-\theta_\alpha(x)\,,\mbox{ \;i.e., \;} \limsup_{t\to\infty} p_A(x,t)\le \theta_\alpha(x)\,,
\end{equation*}
respectively. This implies $p_A(x,t)\to \theta_{\alpha}(x)$ pointwise as $t\to\infty$. By parabolic regularity and the Arzela-Ascoli Lemma, we infer that $p_A(x,t)\to \theta_{\alpha}(x)$ in $C^2(\bar\Omega)$, as in \cite[Theorem 2.1]{LN2002}.
\end{remark}

\begin{remark}\label{remark:thm7.1}\rm Based on Remarks \ref{remark:lem7.5}, \ref{remark:lem7.5b} and the proof of Theorem \ref{thm:7.1}(a), we observe that for every $\delta\in (0, \max\{\la_A,\la_B\})$ the $\epsilon_0$ in Theorem \ref{thm:7.1}(a) can be chosen independently of $\la\in (0,\max\{\la_A,\la_B\}-\delta]$.
\end{remark}

\subsection{Persistence results and existence of internal equilibrium}

For the rest of this paper, we treat the case $\la > \max\{\la_A, \la_B\}$, so that the single-locus problems at loci $\A$ and $\B$ admit linearly stable clines $\theta_\alpha$ and $\theta_\beta$, respectively (Theorem \ref{thm:singlelocus}). First, we will use persistence theory (e.g.\ \cite{ST}) to establish the existence of an internal equilibrium of the two-locus problem.

\begin{definition}
Let $\Phi:\mathbf{Y}\times[0,\infty) \to \mathbf{Y}$ be a semiflow.

\noindent {\rm (i)} $\Phi$ is point-dissipative if there exists $C>0$ independent of initial conditions $Q_0 \in \mathbf{Y}$ such that
\begin{equation}
	\limsup_{t \to \infty} \left\| \Phi_t(Q_0)\right\|_{\mathbf{Y}} \leq C\,.
\end{equation}

\noindent {\rm (ii)} $\Phi$ is eventually bounded on $\mathbf{Y}$ if $\bigcup_{t \geq t_0} \Phi_t(\mathbf{Y})$ is bounded for some $t_0 \geq 0$.

\noindent {\rm (iii)} $\Phi_t: \mathbf{Y} \to \mathbf{Y}$ is compact for given $t>0$ if $\Phi_t(B)$ is precompact for every bounded subset $B$ of $\mathbf{Y}$.
\end{definition}

\begin{proposition}\label{prop:2.2}
The system \eqref{eq:ABD_add} generates a semiflow $\Phi$ on $\mathbf{Y}$, i.e., for initial data $Q_0\in \mathbf{Y}$ and every $t \geq 0$, let $\Phi_t(Q_0)=(p_A(\cdot,t), p_B(\cdot,t), D(\cdot,t))$, where $(p_A,p_B,D)$ is the corresponding solution of \eqref{eq:ABD_add}.
Then $\Phi$ is {\rm (i)} point-dissipative, {\rm (ii)} eventually bounded on $\mathbf{Y}$, and {\rm (iii)} $\Phi_t: \mathbf{Y} \to \mathbf{Y}$ is compact for every $t >0$.
\end{proposition}

\begin{proof}
Because the map  $\mathcal{T}:\mathbf{Y}\to\mathbf{X}$ in \eqref{trafo_T} is a homeomorphism and $\mathbf{X}$ in \eqref{X} is forward invariant under the semiflow $\Psi$ generated by \eqref{dynamics_pi}, $\mathbf{Y}$ is forward invariant under $\Phi$. Therefore, $\Phi_t(Q_0)=(p_A,p_B,D)(\cdot,t)$ exists and remains in $\mathbf{Y}$ for all $t >0$. Since $\mathbf{Y}$ is a bounded set, $\Phi$ is point-dissipative and eventually bounded.

To prove (iii), we rewrite the first two equations of \eqref{eq:ABD_add} as
\begin{align*}
	\partial_t p_A - \Delta p_A = &F_A:= \la \alpha p_A(1- p_A) + \la \beta D\,, \\
	\partial_t p_B - \Delta p_B = &F_B:= \la \beta p_B(1- p_B) + \la \alpha D \,,
\end{align*}
and apply semigroup and regularity theory. For every $t_0 \geq \ta>0$, there exists $C>0$ independent of $\epsilon$ and initial data, such that
$$
\|(p_A,p_B)\|_{W^{2,1,p}(\Omega \times [t_0, t_0+\ta])} \leq C(\|(F_A,F_B)\|_{L^p(\bar\Omega \times [t_0-\ta, t_0+\ta])}+\|(p_A, p_B)\|_{L^p(\bar\Omega\times[t_0-\ta,t_0+\ta])})
$$
\cite[Theorem 7.35]{lieberman}, and the constant $C$ depends on $\min\{t_0, 1\}$ because we can take $\tau = \min\{\tfrac12t_0, \tfrac12\}$.
By Sobolev embedding, we deduce
\begin{align}\label{eq:regular}
	&\sup_{t \in [t_0, t_0+\tau]} \|p_A(\cdot,t), p_B(\cdot,t)\|_{C^{1+\gamma}(\bar\Omega)} \notag\\
	&\qquad\leq C''\left(\|(p_A, p_B)\|_{C(\bar\Omega \times [t_0-\ta, t_0+\ta])} + \|D(\cdot,t)\|_{C(\bar\Omega \times [t_0-\ta, t_0+\ta])} \right)
		\leq \frac{5}{4}C''\,,
\end{align}
where the last inequality follows from \eqref{constraints_twoloc}.

Similarly, for every $t_0 \geq \ta >0$ there is a constant $C_\epsilon$ independent of initial data, such that
$$
	\sup_{t_0 \leq t \leq  t_0 +\ta}\|D(\cdot,t)\|_{C^{1+\gamma}(\bar\Omega)} \leq  C_\epsilon.
$$
(Note that $C_\epsilon$ depends not only on $\min\{t_0,1\}$, as above, but also on $\epsilon$ because the coefficients in equation \eqref{eq:ABD_D} for $D$ depend on $\epsilon$.)
Therefore, $\Phi_t$ is a bounded mapping from $\mathbf{Y} \to \mathbf{Y} \cap C^{1+\gamma}(\bar\Omega; [0,1]^2\times[-\tfrac14,\tfrac14])$ for every $t \geq t_0$, i.e., there is a constant $M_t$, such that $\|\Phi_t(Q_0)\|_{C^{1+\ga}(\bar\Om)}\le M_t$ for all $Q_0\in\mathbf{Y}$. By the compactness of the embedding $ C^{1+\gamma}(\bar\Omega; [0,1]^2\times[-\tfrac14,\tfrac14]) \hookrightarrow C(\bar\Omega; [0,1]^2\times[-\tfrac14,\tfrac14])$ and because $t_0$ can be arbitrarily small, we deduce that
$\Phi_t: \mathbf{Y} \to \mathbf{Y}$ is compact for every $t >0$.
\end{proof}

\begin{corollary}\label{cor:attractor}
The semiflow $\Phi$ has a compact attractor $\mathcal{C}$ of $\mathbf{Y}$, i.e., $\textup{dist}(\Phi_t(\mathbf{Y}),\mathcal{C}) \to 0$ as $t \to \infty$.
\end{corollary}
\begin{proof}
By \cite[Theorem 2.30 and Remark 2.26(b)]{ST}, it is sufficient to verify that the semiflow $\Phi$ is (i) point-dissipative, (ii) eventually bounded on $\mathbf{Y}$, and (iii) $\Phi_t:\mathbf{Y} \to \mathbf{Y}$ is compact for some $t>0$. These have been shown in Proposition \ref{prop:2.2}.
\end{proof}

\begin{definition}
{\rm (i)} Define the function $\kappa: \mathbf{Y} \rightarrow [0,\infty)$ by
\begin{equation}
	\kappa(v_1,v_2,v_3):= \inf_{x \in \Omega}  \left[\min\left\{ v_1(x), 1-v_1(x), v_2(x), 1-v_2(x) \right\}\right]\,.
\end{equation}

\noindent {\rm (ii)} We call the semiflow $\Phi$ uniformly $\kappa$-persistent, if there exists $\delta_0>0$ independent of initial condition $Q_0 \in \mathbf{Y} \setminus \mathbf{Y}_0$ such that
$$
	\liminf_{t \to \infty} \kappa(\Phi_t(Q_0)) = \liminf_{t \to \infty} \left[\inf_{x \in \Omega} \min\left\{p_A(x,t), 1-p_A(x,t), p_B(x,t), 1-p_B(x,t) \right\}\right] \geq \delta_0\,.
$$
\end{definition}

The function $\kappa$ is continuous and, by Lemma~\ref{flow_into_interior}, satisfies $\kappa(p_A(\cdot,t), p_B(\cdot,t), D(\cdot,t))>0$ for $t>0$ if either
$$
	\kappa(p_A(\cdot,0), p_B(\cdot,0), D(\cdot,0))>0
$$
or
$$
	\kappa(p_A(\cdot,0), p_B(\cdot,0), D(\cdot,0))=0 \; \text{ and } \; (p_A(\cdot,0), p_B(\cdot,0), D(\cdot,0)) \in \mathbf{Y}\setminus\mathbf{Y}_0\,.
$$

In the following, we apply standard results from persistence theory to prove the existence of at least one internal equilibrium. Any such equilibrium will satisfy \eqref{eps_estimate_(pA,pB,D)}.

\begin{corollary}\label{cor:existence}
Suppose  $\la > \max\{\la_A, \la_B\}$. Then for every sufficiently small $\epsilon>0$, the system \eqref{eq:ABD_add} has an internal equilibrium, i.e., there exists $(\hat{p}_A, \hat{p}_B, \hat{D}) = (\hat{p}_A^{(\ep)}, \hat{p}_B^{(\ep)}, \hat{D}^{(\ep)})$ in the interior of $\mathbf{Y}$, such that $\kappa(\hat{p}_A, \hat{p}_B, \hat{D}) >0$ and $\Phi_t(\hat{p}_A, \hat{p}_B, \hat{D}) = (\hat{p}_A, \hat{p}_B, \hat{D})$ for all $t \geq 0$.  Moreover,
\begin{equation}\label{eq:corexistence}
	\left\|  (\hat{p}_A, \hat{p}_B) - (\theta_\alpha, \theta_\beta) \right\|_{C^1(\bar\Omega)} + \| \hat{D} \|_{C(\bar\Omega)} = O(\epsilon)
\end{equation}
as $\ep\to0$.
\end{corollary}

\begin{proof}
We recall that the semiflow $\Phi$ on $\mathbf{Y}$ is equivalent to the semiflow $\Psi$ on $\mathbf{X}$ via the relation
$\Phi_t = \mathcal{T}^{-1} \circ \Psi_t \circ \mathcal{T}$, where $\mathbf{X}$ is given in $\eqref{X}$, and $\mathcal{T}(p_A, p_B,D) = (p_1,p_2,p_3,p_4)$ is given in \eqref{trafo_T}.
If we define $\kappa': \mathbf{X} \rightarrow [0,\infty)$ by
$$
	\kappa'(u_1,u_2,u_3,u_4)=  \inf_{x \in \Omega} \left[ \min\left\{  u_1+u_2, 1-  u_1-u_2, u_1+u_3, 1- u_1-u_3 \right\}  \right]\,,
$$
then $\kappa' = \kappa\circ \mathcal{T}^{-1}$.

For every fixed, sufficiently small $\epsilon$, we observe that (i) the semiflow $\Psi$  is uniformly $\kappa'$-persistent (because $\Phi$ is uniformly $\kappa$-persistent by Proposition \ref{prop:persistence}); (ii) $\Psi_t: \mathbf{X} \to \mathbf{X}$  is compact, hence condensing, for every $t >0 $ (because $\Phi_t: \mathbf{Y} \to \mathbf{Y}$  is compact for every $t>0$ by Proposition \ref{prop:2.2}, and $\mathcal{T}:\mathbf{Y} \to \mathbf{X}$ is a homeomorphism); and (iii) $\Psi$ has a compact attractor in $\mathbf{X}$ (because $\Phi$ has a compact attractor in $\mathbf{Y}$ by Corollary \ref{cor:attractor}), which shows that $\Psi$ has a compact attractor of neighborhoods of compact sets.

Observe in addition that
\begin{itemize}

\item $\mathbf{X}$ is a closed convex subset of the Banach space $C(\bar\Omega; \mathbb{R}^4)$.

\item $\kappa': \mathbf{X} \to \mathbb{R}_+$ is continuous and concave, where concave means
$$
	\kappa'(\la Q_1 + (1-\la) Q_2) \geq \la \kappa'(Q_1) + (1-\la) \kappa'(Q_2)
$$
for all $\la \in [0,1]$ and $Q_1, Q_2 \in \mathbf{X}$.
\end{itemize}

Therefore, the existence of an equilibrium $(\hat{p}_1,\hat{p}_2, \hat{p}_3,\hat{p}_4)$ satisfying $\kappa'(\hat{p}_1,\hat{p}_2, \hat{p}_3,\hat{p}_4) >0$ follows from  \cite[Theorem 6.2]{ST}. Hence, $(\hat{p}_A,\hat{p}_B,\hat{D}) := \mathcal{T}^{-1} (\hat{p}_1,\hat{p}_2, \hat{p}_3,\hat{p}_4)$
is an equilibrium of the semiflow $\Phi$ associated with \eqref{eq:ABD_add}. Because
$$
\kappa(\hat{p}_A,\hat{p}_B,\hat{D}) = \kappa'(\hat{p}_1,\hat{p}_2, \hat{p}_3,\hat{p}_4) >0\,,
$$
$(\hat{p}_A,\hat{p}_B,\hat{D})$ is an internal equilibrium of \eqref{eq:ABD_add}.  Finally, \eqref{eq:corexistence} follows from \eqref{eq:prop2a}, \eqref{eq:prop2b}, and \eqref{eq:cor7.4b}.
\end{proof}

\subsection{Global asymptotic stability}

Let $\la > \max\{\la_A, \la_B\}$ and let $(\hat{p}_A, \hat{p}_B, \hat{D})$ be an internal equilibrium given by Corollary \ref{cor:existence}, we will show that  it attracts all trajectories initiating in $\mathbf{Y} \setminus \mathbf{Y}_0$. This in particular implies the uniqueness of the internal equilibrium. Part (b) of Theorem \ref{thm:7.1} is an immediate consequence of Corollary~\ref{cor:existence} and the following proposition.

\begin{proposition}\label{prop:final}
Let $\la > \max\{\la_A, \la_B\}$.
For every sufficiently small $\epsilon>0$, the internal equilibrium $(\hat{p}_A, \hat{p}_B, \hat{D})$ attracts all trajectories initiating in $\mathbf{Y} \setminus \mathbf{Y}_0$, where convergence occurs in
$[C^2(\bar\Om)]^3$.
In particular, $(\hat{p}_A, \hat{p}_B, \hat{D})$ is the unique internal equilibrium of \eqref{eq:ABD_add}.
\end{proposition}

To prepare for the proof of Proposition \ref{prop:final}, we define
$$
(\tilde{p}_A(x,t), \tilde{p}_B(x,t), \tilde{D}(x,t)):= (p_A(x,t) - \hat{p}_A(x), p_B(x,t) - \hat{p}_B(x), D(x,t) - \hat{D}(x)).
$$
If $\la > \max\{\la_A, \la_B\}$, then by Proposition~\ref{prop:persistence}, Remark~\ref{rem:D_to_O(eps)}, and Corollary~\ref{cor:existence}, there exist $C_1>0$ and $\ep_1>0$ such that
\begin{equation}\label{eq:ultimate}
\limsup_{t \to \infty}\left[  \left\|(\tilde{p}_A, \tilde{p}_B)(\cdot,t) \right\|_{C^1(\bar\Omega)}  + \|\tilde{D}(\cdot,t) \|_{C(\bar\Omega)} \right]\leq C_1 \epsilon
\end{equation}
for every $\ep\in(0,\ep_1]$.
Furthermore, observe that $(\tilde{p}_A(x,t), \tilde{p}_B(x,t), \tilde{D}(x,t))$ satisfies
\begin{subequations}\label{eq:subtractp}
\begin{alignat}{3}
		&\partial_t \tilde{p}_A - \Delta \tilde{p}_A - \la \alpha(x) (1- 2\hat{p}_A(x))\tilde{p}_A
		 = -\la\alpha(x)(\tilde{p}_A)^2+ \la \beta(x) \tilde{D} &&\quad\text{ in } \Omega\times(0,\infty)\,, \\
		&\partial_t \tilde{p}_B - \Delta \tilde{p}_B - \la \beta(x) (1- 2\hat{p}_B(x))\tilde{p}_B = -\la \beta(x)(\tilde{p}_B)^2+  \la\alpha(x) \tilde{D} &&\quad\text{ in } \Omega\times(0,\infty)\,, \\
		&\partial_\nu \tilde{p}_A = \partial_\nu \tilde{p}_B = 0 &&\quad\text{ on } \partial\Omega\times(0,\infty)\,,
\end{alignat}
\end{subequations}
and
\begin{subequations}
\begin{alignat}{2}
	&\partial_t \tilde{D} - \Delta \tilde{D}  -\la \left[ \alpha (1-2p_A) + \beta (1-2p_B)\right]\tilde{D} + \frac{1}{\epsilon}\tilde{D}  \notag  \\
 &\qquad =  2\nabla p_B \cdot \nabla \tilde{p}_A + 2 \nabla \hat{p}_A \cdot \nabla \tilde{p}_B - 2\la \alpha \hat{D} \tilde{p}_A - 2\la \beta \hat{D} \tilde{p}_B &&\quad\text{in } \Omega\times(0,\infty)\,, \label{PDE_tilde D} \\
&\partial_\nu \tilde{D} = 0 &&\quad\text{on } \Omega\times(0,\infty)\,.
\end{alignat}
\end{subequations}

\begin{lemma}
Let $\la > \max\{\la_A, \la_B\}$. Then there exists $C_2>0$ such that
\begin{equation}\label{eq:step3b}
	\limsup_{t \to \infty} \|\tilde{D}(\cdot,t)\|_{C(\bar\Omega)}  \leq \epsilon C_2 \limsup_{t \to \infty} \|(\tilde{p}_A(\cdot,t), \tilde{p}_B(\cdot,t))\|_{C^1(\bar\Omega)}
\end{equation}
for every $\ep \le \min\{\ep_1,\ep_2\}$, where $\epsilon_1$ is associated with \eqref{eq:ultimate} and $\epsilon_2$ is chosen such that
\begin{equation}
	\frac{1}{2\epsilon_2} =  \la\sup_{x \in \Omega} \left( |\alpha(x)| + |\beta(x)|\right) \geq \la  \sup_{x \in \Omega} |\alpha(x)(1-2p_A(x)) + \beta(x)(1-2p_B(x))| \,. \label{1/eps_bound2}
\end{equation}
\end{lemma}

\begin{proof}
To prove \eqref{eq:step3b}, we define
$$
\tilde{D}^*(t)=\max\{\sup_{x \in \Omega} \tilde{D}(x,t),0\} \;\text{ and }\; \tilde{p}^*(t) = \|(\tilde{p}_A(\cdot,t), \tilde{p}_B(\cdot,t))\|_{C^1(\bar\Omega)} \,.
$$
We choose, by \eqref{eq:lem7.4-2}, a positive constant $C_2>0$ such that
 the right hand side of \eqref{PDE_tilde D} is bounded from above by $\frac12C_2\tilde{p}^*(t)$ for $t \geq 1$. We claim that $\tilde{D}^*$ satisfies the following differential inequality (in the weak sense)
 \begin{equation}\label{eq:weakdiffeq}
	\frac{d}{dt}\tilde{D}^*(t) + \frac{1}{2\epsilon} \tilde{D}^*(t) \leq \frac{C_2}{2} \tilde{p}^*(t) \quad \text{ for }t \in (1,\infty).
\end{equation}
First, we observe that  $\tilde{D}^*(t)$ is Lipschitz in $[1,\infty)$. For fixed $M>0$ and $t_1, t_2 \in [1,M]$,
we assume without loss that $\tilde{D}^*(t_1) \leq  \tilde{D}^*(t_2)$, and let $\sup_{x \in \Omega} \tilde{D}(x,t_i) = \tilde{D}(x_i,t_i)$ for some $x_i \in \bar\Omega$ ($i=1,2$).  Then
\begin{align*}
	|\tilde{D}^*(t_2) - \tilde{D}^*(t_1)|  &= \max\{\tilde{D}(x_2,t_2),0\} - \max\{\tilde{D}(x_1,t_1),0\} \\
			&\leq \tilde{D}(x_2,t_2) - \tilde{D}(x_1,t_1) \leq \tilde{D}(x_2,t_2) - \tilde{D}(x_2, t_1)
\end{align*}
and thus $[\tilde{D}^*]_{\rm{Lip}([1,M])} \leq \|\partial_t\tilde{D}\|_{C(\bar\Omega \times [1,M])}$, where the latter is finite because $\partial_t\tilde{D}$ is H\"older continuous by parabolic Schauder estimates.

It remains to show that $\tilde{D}^*$ satisfies \eqref{eq:weakdiffeq} whenever it is differentiable. To this end, suppose $\frac{d}{dt}\tilde{D}^*(t_0)$ exists for some $t_0>0$. There are two cases: Case (a) $\sup_{x \in \Omega} \tilde{D}(x,t_0) <0$; Case (b) $\tilde{D}^*(t_0) = \tilde{D}(x_0,t_0) \geq 0$ for some $x_0 \in \bar\Omega$. In Case (a), $\tilde{D}^*(t) = 0$ in a neigborhood of $t_0$ and \eqref{eq:weakdiffeq} holds trivially. For Case (b), we claim that $\Delta \tilde{D}(x_0,t_0) \leq 0$. Assume not, then $\Delta \tilde{D}(x_0,t_0) >0$ and $x_0$ cannot be an interior maximum point. Thus, $x_0 \in \partial\Omega$ and there exists $\delta'>0$ such that
$$
\tilde{D}(x,t_0) < \tilde{D}(x_0,t_0) \,\,\text{ and }\quad \Delta \tilde{D}(x,t_0) >0\quad  \text{in }B_{\delta'}(x_0) \cap \bar\Omega.
$$
But then the Hopf lemma applies to yield that $\partial_\nu \tilde{D}(x_0,t_0) >0$. This is in contradiction with the Neumann boundary condition imposed on $\tilde{D}$ on $\partial\Omega \times (0,\infty).$ Thus, $\Delta \tilde{D}(x_0,t_0) \leq 0$.

With this, we may evaluate \eqref{PDE_tilde D} at $(x_0,t_0)$ to obtain (here the choice of $\epsilon < \epsilon_2$ is needed) 
$$
\frac{\partial}{\partial t} \tilde{D}(x_0,t_0)+ \frac{1}{2\epsilon} \tilde{D} (x_0,t_0)\leq \frac{C_2}{2} \tilde{p}^*(t_0).
$$
Since $\tilde{D}^*$ is differentiable at $t_0$ and $\tilde{D}^*(t_0)=\tilde{D}(x_0,t_0)\geq 0$, we must have $\frac{d}{dt}\tilde{D}^*(t_0) = \frac{\partial}{\partial t} \tilde{D}(x_0,t_0)$, hence we deduce \eqref{eq:weakdiffeq} at $t= t_0$. Since $\tilde{D}^* \in C([0,\infty)) \cap {\rm Lip}\,([1,\infty))$ (and thus absolutely continuous in $[1,\infty)$), and satisfies \eqref{eq:weakdiffeq} at all points where it is differentiable, it satisfies \eqref{eq:weakdiffeq} in the weak sense.

From \eqref{eq:weakdiffeq} we deduce
$$
	\tilde{D}^*(t) \leq \tilde{D}^*(1) e^{-\frac{(t-1)}{2\epsilon}} + \frac{C_2}{2} \int_1^t e^{\frac{-(t-s)}{2\epsilon}} \tilde{p}^*(s)\,ds \, \quad \text{ for }t\geq 1.
$$
This implies
$$
\limsup_{t \to \infty} \left[\sup_{x \in \Omega} \tilde{D}(x,t)\right] \leq  \limsup_{t \to \infty} \tilde{D}^*(t) \leq C_2\epsilon \limsup_{t \to \infty}  \tilde{p}^*(t).
$$
Similarly, we obtain
$$
\liminf_{t \to \infty}  \left[\inf_{x \in \Omega} \tilde{D}(x,t)\right] \geq -C_2\epsilon \liminf_{t \to \infty}  \tilde{p}^*(t)\,,
$$
which proves \eqref{eq:step3b}.
\end{proof}

We are now in the position to prove the main result of this section.

\begin{proof}[Proof of Proposition \ref{prop:final}]
We claim that
\begin{equation}\label{eq:step3a}
\limsup_{t \to \infty} \|(\tilde{p}_A,\tilde{p}_B)(\cdot,t)\|_{C^1(\bar\Omega)} = 0.
\end{equation}
To this end, let $L_{\hat p_A}$ and $L_{\hat p_B}$ be defined according to \eqref{def_L_ph}. By \eqref{eq:corexistence}, we can apply Lemma \ref{lem:U1} and obtain
$$
	\sigma(L_{\hat p_A}) \subset \{z \in \mathbb{C}: \textup{Re}\, z > \delta_0\} \; \text{ and }\, \sigma(L_{\hat p_B}) \subset \{z \in \mathbb{C}: \textup{Re}\, z > \delta_0\}
$$
for some $\delta_0>0$. Hence, we can apply Lemma \ref{lem:U2} to \eqref{eq:subtractp} to deduce
$$
\limsup_{t \to \infty} \|(\tilde{p}_A, \tilde{p}_B)(\cdot,t)\|_{C^1(\bar\Omega)}  \leq C_3 \left[ \left(\limsup_{t \to \infty} \|(\tilde{p}_A,\tilde{p}_B)(\cdot,t)\|_{C^1(\bar\Omega)}\right)^2 + \limsup_{t \to \infty} \|\tilde{D}(\cdot,t)\|_{C(\bar\Omega)} \right] \,.
$$
Now, by \eqref{eq:ultimate} and \eqref{eq:step3b} there exists a constant $C_4$ independent of $\epsilon$ such that
$$
	\limsup_{t \to \infty} \|(\tilde{p}_A, \tilde{p}_B)(\cdot,t)\|_{C^1(\bar\Omega)}  \leq C_4 \epsilon \left[ \limsup_{t \to \infty} \|(\tilde{p}_A,\tilde{p}_B)(\cdot,t)\|_{C^1(\bar\Omega)}\right]\,.
$$
This proves \eqref{eq:step3a} provided $\epsilon < \min\{\ep_1,\ep_2, 1/C_4\}$.

Finally, the estimates \eqref{eq:step3a} and \eqref{eq:step3b} imply
$$
\limsup_{t \to \infty} \|(\tilde{p}_A, \tilde{p}_B)(\cdot,t)\|_{C^1(\bar\Omega)}  = \limsup_{t \to \infty} \|\tilde{D}(\cdot,t)\|_{C(\bar\Omega)}  = 0\,,
$$
i.e., $(p_A(\cdot,t), p_B(\cdot,t), D(\cdot,t)) \to (\hat{p}_A, \hat{p}_B, \hat{D})$ in $C^1(\bar\Omega) \times C^1(\bar\Omega) \times C(\bar\Omega)$ as $t\to\infty$. In particular, $ (\hat{p}_A, \hat{p}_B, \hat{D})$ is the unique internal equilibrium of \eqref{eq:ABD_add}. As before, for each fixed $\epsilon>0$, we may apply parabolic regularity theory to strengthen the above convergence to $[C^2(\bar\Omega)]^3$.   This completes the proof.
\end{proof}

\section{Discussion}\label{sec:Disc}
The aim of this work was the establishment of conditions for existence, uniqueness, and stability of two-locus clines. This has been achieved for two limiting cases: weak recombination ($\rh\ll1$, Theorem \ref{thm:weak_reco}) and strong recombination ($\rh\gg1$, Theorem \ref{thm:7.1}). In the latter case, even global asymptotic stability could be proved, whereas in the former case only existence and linear stability were proved. For general strength of recombination, the problem remains largely unresolved, and the equilibrium structure and dynamics are likely more complex.

We conjecture that for intermediate recombination rates and if the strength of selection relative to diffusion is in a certain range, an internal equilibrium, i.e., a two-locus cline, can be simultaneously stable with a boundary equilibrium. For a related ODE model, in which there is unidirectional migration from one deme into an other deme, this was proved in \cite{BuergerAkerman2011}. Numerical solution of the system \eqref{eq:ABD_add} supports this conjecture (RB, unpublished).

A global convergence result that applies to arbitrary recombination is Theorem \ref{thm:M1_globallystable}. It shows that for every fixed $r\ge 0$ and $s>0$, there exists $d_0=d_0(r,s)\gg1$ such that the monomorphic equilibrium with the highest spatially averaged fitness is globally asymptotically stable if $d>d_0$. We conjecture that for given $s>0$, $d_0$ can be chosen even independent of $r\ge0$; in other words, there exists $\la_0\ll 1$ such that such monomorphic equilibrium is globally asymptotically stable for \eqref{dynamics_pi} if $\la<\la_0$ (Remark \ref{rem:4.7}).

A limiting case, for which we also conjecture existence of a globally asymptotically stable two-locus cline is that of weak migration relative to selection and recombination ($d\ll1$). However, this limit is degenerate. For a single locus, profiles of the clines were derived in this limit under various assumptions about dominance in \cite{LN2004} and \cite{NNS2010}. There are other cases that should be amenable to a rigorous analysis.

For a finite number of demes, several limiting cases were studied rigorously in \cite{RB2009}. In such discrete-space models, selection and recombination in each deme are described by difference equations (if generations are discrete) or ODEs (if generations are overlapping), and migration between demes is modeled by an ergodic matrix. In \cite{RB2009}, global convergence results were proved for weak migration and for strong migration, subject to additional, also scaling, assumptions which guaranteed that the set of chain-recurrent points of an appropriate limiting system consists of hyperbolic equilibria only. There, an arbitrary number of multiallelic loci was admitted as well as selection schemes with dominance and epistasis. Despite this additional complications (which enable multiple stable equilibria), the proofs (if restricted to two diallelic loci) are simpler than here, and also different, because they rely on methods and results developed in \cite{NHB1999} and invoke perturbation theory of compact normally hyperbolic manifolds and of chain-recurrent sets for dynamical systems on compact state spaces. The case of strong recombination is briefly outlined in \cite[Section 7.9]{RB2014}. For the special case of two diallelic loci and additive fitnesses as in \eqref{s1234} and \eqref{S_i}, the result given there reduces to an analogue of the present Theorem~\ref{thm:7.1}.

It is of considerable biological interest to study how the shape of a cline depends on the underlying parameters. In the present context, population genetic intuition suggests that the two-locus cline becomes steeper with stronger linkage, i.e., smaller $r$ (hence $\rh$), provided the functions $\a$ and $\be$ have the same sign. The reason is that positive linkage disequilibrium (covariance) between the loci will be generated in this case, so that a kind of mutual reinforcement emerges. Support for this conjecture comes from numerical results and formal calculations \cite{Barton1983,Barton1986,Slatkin1975}, as well as from related ODE models  \cite{AkermanBuerger2014,BuergerAkerman2011,GeroldingerBuerger2015}. For a step environment on the real line, i.e., if each of $\a(x)$ and $\be(x)$ assume only two values and change sign at the same location, the slope of each of the allele-frequency clines ($p_A$, $p_B$) at the step was shown to increase with decreasing $\rh$ provided $\rh$ was sufficiently large \cite{RB2017}. This was done by deriving an explicit first-order approximation of the two-locus cline. It would be of interest to show similar results for the allele-frequency clines of the present model, possibly following \cite{LL2011} and using $||\nabla p_A||_{L^2(\Om)}$ as a measure of the steepness.

Throughout the present paper, we assumed an open bounded domain. It would be desirable and challenging to develop an analogous theory for unbounded domains. For one locus with two alleles, various results on the existence, uniqueness and stability of clines were derived in \cite{Conley1975} and \cite{Fife&Peletier1977}. In particular, Conley \cite{Conley1975} showed that a cline exists if the function describing the influence of environmental variation, say $h(x)$, is not integrable near $\pm\infty$ and $\sgn h(x) = \sgn x$. Therefore, in contrast to a bounded domain, a cline exists independently of the strength of diffusion relative to selection (see also \cite{RB2017} for the two-locus model with a step environment). For the two-locus case, one may conjecture that a two-locus cline exists if both $\a(x)$ and $\be(x)$ satisfy these conditions on $h(x)$.

Another general assumption was that the functions $\a(x)$ and $\be(x)$ change sign in $\Om$, i.e., (A). For a single locus, it is well known that in the absence of a sign change, one of the trivial equilibria is globally asymptotically stable (eg.\ \cite{LN2002,LNN2013}). Assume that $\be(x)$ does not change sign, but $\a(x)$ does. Then the results in Section \ref{sec:ev-problems} imply that $\la_B=\la_\be=\infty$. Therefore, we can follow the proof of Theorem \ref{thm:7.1}(a) to show global convergence to a boundary equilibrium for every $\la>0$.

In Theorem \ref{thm:7.1}, the degenerate case $\la=\max\{\la_A,\la_B\}$ was excluded. Assuming $\la=\la_A > \la_B$,
our results in Section \ref{sec:strong_reco} show that $\th_\a=0$ (or $\th_\a=1$) and $0<\th_\be<1$. Straightforward linearization is insufficient to determine whether the perturbation of the equilibrium $(\th_\a,\th_\be,0)$ is in the state space or not. We expect that a sufficient condition for the existence of an internal equilibrium for large $\rh$ is that $\a(x)$ and $\be(x)$ have the same sign.

{\bf Acknowledgments.} The authors gratefully acknowledge helpful discussions with Profs.\ Josef Hofbauer and Yuan Lou, inspiring communication with Prof.\ Thomas Nagylaki, and useful comments by T.\ Nagylaki and an anonymous reviewer. LS and RB were supported by the Austrian Science Fund (FWF) through Grant P25188-N25. LS was also supported by the National Natural Science Foundation of China (NSFC), Grant 11501283.

\end{document}